\documentclass[USenglish]{article}

\usepackage[big,online]{dgruyter}	%values: small,big | online,print,work
\usepackage{lmodern} 
\usepackage{microtype}
\usepackage[numbers,square,sort&compress]{natbib}

\usepackage{amssymb, amsmath, amsfonts}
\usepackage{hyperref}
\hypersetup{
	colorlinks=true,
	linkcolor=red,
	citecolor=blue,
	filecolor=cyan,      
	urlcolor=magenta}
\usepackage{bm}

\usepackage[all]{xy}
\usepackage{tikz-cd} 

\newcommand\genfd{{\bm k}}
\newcommand\op{\mathrm{op}}

\newcommand\id{\mathrm{id}}

\newcommand\Ug{U(\mathfrak{g})}
\newcommand\OO{\mathcal{O}}
\newcommand\Aut{\operatorname{Aut}}
\newcommand\ggf{\mathfrak{g}}

\newcommand\ad{\operatorname{ad}}

\newcommand\GG{\mathcal{G}}

\newcommand\Hom{\operatorname{Hom}}

\newcommand\Vect{\operatorname{Vec}_\genfd}

\usepackage{mathabx}

\newcommand\btl{\blacktriangleleft}

\usepackage{amsthm}
\newtheoremstyle{definition}{}{}{\upshape}{}{\bfseries}{.}{0.5em}{}

\newtheorem{theorem}{Theorem}[section]
\newtheorem{lemma}[theorem]{Lemma}

\newtheorem{proposition}[theorem]{Proposition}
\theoremstyle{definition}

\newtheorem{remark}[theorem]{Remark}

\begin{document}

\title{Enveloping algebra is a Yetter--Drinfeld module algebra over Hopf algebra of regular functions on the automorphism group of a Lie algebra}
\runningtitle{Enveloping algebra is a Yetter--Drinfeld module algebra}
\author*[1]{Zoran \v{S}koda}
\author[2]{Martina Stoji\'c}
\runningauthor{Z. \v{S}koda, M. Stoji\'c}
\affil[1]{\protect\raggedright Department of Teachers’ Education, University of Zadar, Franje Tudjmana 24, 23000 Zadar, Croatia, e-mail: zskoda@unizd.hr}
\affil[2]{\protect\raggedright Department of Mathematics, University of Zagreb, Bijeni\v{c}ka cesta~30, 10000 Zagreb, Croatia, e-mail: stojic@math.hr}
      %  \date{}

\abstract{
We present an elementary construction of a (highly degenerate) Hopf pairing
between the universal enveloping algebra $U(\ggf)$ of a finite-dimensional Lie algebra $\ggf$ over arbitrary field $\genfd$ and the Hopf algebra $\OO(\Aut(\ggf))$ of regular functions on the automorphism group of $\ggf$. This pairing induces a Hopf action of $\OO(\Aut(\ggf))$ on $U(\ggf)$ which together with an explicitly given coaction makes $U(\ggf)$ into a braided commutative Yetter--Drinfeld $\OO(\Aut(\ggf))$-module algebra. From these data one constructs a Hopf algebroid structure on the smash product algebra $\OO(\Aut(\ggf))\sharp U(\ggf)$ retaining essential features from earlier constructions of a Hopf algebroid structure on infinite-dimensional versions of Heisenberg double of $U(\ggf)$, including a noncommutative phase space of Lie algebra type, while avoiding the need of completed tensor products.

We prove a slightly more general result where algebra $\OO(\Aut(\ggf))$ is replaced by $\OO(\Aut(\mathfrak{h}))$ and where $\mathfrak{h}$ is any finite-dimensional Leibniz algebra having $\ggf$ as its maximal Lie algebra quotient.
}


\maketitle
 
\section{Introduction}

Yetter--Drinfeld modules
over bialgebras~(\ref{ss:YD} and~\cite{radford,radfordtowber,SS:two})
are ubiquitous in quantum algebra, low dimensional
topology and representation theory. More intricate structure of
a braided commutative monoid in the category ${}^H\mathcal{YD}_H$
of Yetter--Drinfeld modules is an ingredient in a construction of 
scalar extension bialgebroids and Hopf algebroids.
Namely, given a Hopf algebra $H$ and a braided commutative 
Yetter--Drinfeld $H$-module algebra $A$, the smash
product algebra $H\sharp A$ has a structure of
a Hopf algebroid over $A$~\cite{bohmHbk,BrzMilitaru,stojic}.
An important special case 
is the Heisenberg double $A^*\sharp A$ of a finite-dimensional
Hopf algebra $A$, where $A^*$ is the dual Hopf algebra of~$A$; the $A^*$-coaction
on $A$ is given by an explicit formula involving basis of $A$ and the dual
basis of $A^*$ (\cite{Lu}, map $\beta$ in Section 6). 
There are several important examples in literature,
some motivated by mathematical physics, when the underlying
algebra of $A$ is the universal enveloping algebra $U(\ggf)$ of
an $n$-dimensional Lie algebra $\ggf$. In the example of a noncommutative phase space $\mathcal{H}_\ggf$ of Lie algebra type~\cite{halg}, the Hopf algebra $H$ is a realization via formal power series of the algebraic dual $U(\ggf)^*$ with its natural topological Hopf algebra structure. Historically, $\mathcal{H}_\ggf$ as an algebra has been introduced by extending $U(\ggf)$ by deformed derivatives (``momenta'') in~\cite{leib} and in a rather nonrigorous treatment~\cite{heisd} it has been argued that $\mathcal{H}_\ggf$ is actually the Heisenberg double of $U(\ggf)$; this made plausible that Hopf algebroid structure on $\mathcal{H}_\ggf$ could be exhibited analogously to Lu's example of finite-dimensional double, leading to an ad hoc version of completed Hopf algebroid structure in~\cite{halg}. An abstract version of $U(\ggf)^*\sharp U(\ggf)$ is described in~\cite{stojicphd} as an internal Hopf algebroid in the symmetric monoidal category of filtered cofiltered vector spaces. These examples may be viewed as
infinite-dimensional cases of Heisenberg double (in fact, this observation from~\cite{heisd} influenced the Hopf algebroid approach in~\cite{halg}), but a number of results in~\cite{stojicphd} show that there are intricate conditions for which infinite-dimensional dually paired Hopf algebras $H$ and $A$ one can indeed form a Hopf algebroid structure on the smash product algebra $H\sharp A$, even in a completed sense or even when one of the Hopf algebras is the restricted dual of another.
In the case of $U(\ggf)$, formulas from~\cite{halg} show
that there are special elements $\mathcal{O}^i_j\in H=U(\ggf)^*$ for $i,j\in\{1,\ldots,n\}$ such that the coaction on the generators is given by the formulas
(adapted to our conventions) $U(\ggf)\supset\ggf\ni x_k\mapsto \sum_i \mathcal{O}^i_k\otimes x_i\in U(\ggf)^*\otimes\ggf$. 
Moreover, elements~$\mathcal{O}^i_j$ satisfy
the relations which are satisfied by matrix elements of automorphisms of $\ggf$.
The prime motivation of this article is to demystify this phenomenon
and to find a much smaller Hopf algebra $H$ containing an abstract
model for $\mathcal{O}^i_j$ and avoiding any completions in describing
Yetter--Drinfeld $H$-module structure on~$U(\ggf)$.

We show that the Hopf algebra of regular functions $H = \mathcal{O}(\Aut(\ggf))$ on the automorphism group of~$\ggf$ will do, namely that $U(\ggf)$ is a braided commutative Yetter--Drinfeld module $\mathcal{O}(\Aut(\ggf))$-algebra whose structure is given by essentially the same formulas as in the case of $U(\ggf)^*$
from~\cite{stojicphd}. General formulas for scalar extensions from~\cite{stojic} describe the Hopf algebroid structure on the smash product $\mathcal{O}(\Aut(\ggf))\sharp U(\ggf)$. 
We write formulas for the symmetric Hopf algebroid structure in full detail. In the work~\cite{omin} we present several other natural examples of Hopf algebras $H$ equipped with a Hopf algebra homomorphism $\mathcal{O}(\Aut(\ggf))\to H$
and where $U(\ggf)$ is still a Yetter--Drinfeld module algebra over $H$ without completions.

The main result of this article is actually
proved in a slightly more general form than
described above.
Namely, instead of the automorphism group of a Lie algebra $\ggf$ we can take
the automorphism group of any
finite-dimensional Leibniz algebra $\mathfrak{h}$
such that $\ggf$ is the maximal Lie algebra quotient
$\mathfrak{h}_{{Lie}}$ of $\mathfrak{h}$, enabling a structure of a Hopf algebroid on the smash product $\OO(\Aut(\mathfrak h))\sharp U(\mathfrak{h}_{Lie})$.

\section{Preliminaries}

\subsection{General conventions and preliminaries on pairings}\label{ss:pairing}
Throughout the paper we freely use Sweedler notation with or without the summation sign~\cite{majid,radford} and the Kronecker symbol $\delta^i_j$. Throughout, $\genfd$ is a fixed ground field and $\Vect$ the category of $\genfd$-vector spaces. If $V\in\Vect$, $V^*:=\Hom_\genfd(V,\genfd)$. If $\Delta$ is a comultiplication on a $\genfd$-coalgebra $C$ then for $m\geq 1$, $\Delta^m := (\id_{C^{\otimes m-2}}\otimes\Delta)\circ\ldots\circ(\id_C\otimes\Delta)\circ\Delta\colon C\to C^{\otimes m+1}$.
If $(A,\mu,\eta)$ is an associative $\genfd$-algebra with multiplication $\mu\colon A\otimes A\to A$ and unit map $\eta$, consider its transpose $\mu^*\colon A^*\to (A\otimes A)^*$. \emph{Restricted dual} $A^\circ\subset A^*$ consists of all $f\in A^*$ such that $\mu^*(f)$ falls within the image of inclusion $A^*\otimes A^*\hookrightarrow (A\otimes A)^*$. It follows that $\mu^*(f)$ belongs also to the image of $A^\circ\otimes A^\circ$ hence the restriction $\mu^*(f)|_{A^\circ}$ may be corestricted to a map $\Delta_{A^\circ}\colon A^\circ\to A^\circ\otimes A^\circ$ making $A^\circ$ into a coalgebra; if $A$ is a bialgebra  (resp.\ Hopf algebra) then $A^\circ$ is. 
\emph{Pairings} of vector spaces are bilinear maps into the ground field which are in this work not required to be nondegenerate. For $V,W\in\Vect$, a pairing $\langle -,-\rangle\colon V\otimes W\to\genfd$ induces a pairing between $V\otimes V$ and $W\otimes W$ componentwise: $\langle v\otimes v',w\otimes w'\rangle :=
\langle v,w\rangle \cdot_\genfd \langle v',w'\rangle$. If $V\in\Vect$ and $A$ is an algebra, and $f\in A^*$ such that there exists an element $f'\in A^*\otimes A^*$ such that $\langle f',g\otimes h\rangle = \langle f, g\cdot h\rangle$ for all $g,h\in A$, then the functional $\langle f, -\rangle\in A^\circ$ and $\langle f',-\rangle\in A^\circ\otimes A^\circ$. If $V = C$ is a coalgebra then $\langle\Delta(f),-\rangle\in A^*\otimes A^*$. Thus, if $\langle\Delta(c),h\otimes g\rangle = \langle c, g\cdot h\rangle$ for all $g,h\in A$, then $\phi_1\colon c\mapsto \langle c,-\rangle :=\phi_1(c)$ is a map $C\to A^\circ$, and $(\phi_1\otimes\phi_1)\circ\Delta_C = \Delta_{A^\circ}\circ\phi_1$. It is a coalgebra map if moreover $\langle c,1\rangle = \epsilon(c)$. Conversely, if $c\mapsto\langle c,-\rangle$ corestricts to a coalgebra map $\phi_1\colon C\to A^\circ$, then the identities $\langle\Delta(c),h\otimes g\rangle = \langle c, g\cdot h\rangle$ and $\langle c,1\rangle = \epsilon(c)$ hold for all $c\in C$, $g,h\in A$. Both conditions hold if and only if $a\mapsto\langle -,a\rangle$ is a map of algebras $A\to C^*$.
A pairing between two bialgebras $B$ and $H$ is Hopf if
$\langle \Delta_B(b), h\otimes k\rangle = \langle b, h\cdot k\rangle$, $\epsilon_B(b) = \langle b,1_H\rangle$ and the symmetric conditions $\langle b\otimes c, \Delta_H(h)\rangle = \langle b\cdot c,h\rangle$, $\epsilon_H(h) = \langle 1_B,h\rangle$. Clearly, the latter two conditions hold if and only if $h\mapsto\langle -,h\rangle$ corestricts to a coalgebra map $\phi_2\colon H\to B^\circ$. Alternatively, the pairing is Hopf if and only if $b\mapsto\langle b,-\rangle$ corestricts to a bialgebra map $\phi_1\colon B\to H^\circ$. If $B$ and $H$ are Hopf algebras, $\phi_1$ is a bialgebra map between Hopf algebras, hence it automatically respects the antipode. Thus, for every Hopf pairing between Hopf algebras, identity  $\langle S_B(b), h \rangle = \langle b, S_H(h)\rangle$ holds for all $b\in B$, $h\in H$.

Replacing Hopf pairing with a bialgebra map $\phi_1\colon B\to H^\circ$ is useful in constructing new pairings from old. Namely, if $I\subset B$ is a \emph{biideal} (ideal which is also a \emph{coideal} in the sense that $\Delta(I)\subset I\otimes B+ B\otimes I$ and $\epsilon(I) = 0$), then there is an induced bialgebra map $\phi_1\colon B/I\to H^\circ$ if and only if $\phi_1(I)= 0$ that is $\langle i,h\rangle = 0$ for all $i\in I, h\in H$.

\begin{lemma}\label{lem:coidealK} Let $\langle -,-\rangle\colon B\otimes H\to\genfd$ be a bialgebra pairing.
\begin{enumerate} 
\item[(i)] For the bialgebra pairing to vanish on $I\otimes H$ where $I\subset B$ is a coideal,
it suffices that it vanishes on $I\otimes K_H$ where $K_H$ is some set of algebra generators of $H$.

\item[(ii)] If $I\subset B$ is a biideal, $K_I\subset I$ a set of generators of $I$ as an ideal and $K_H$ a set of generators of $H$ as an algebra, then the bialgebra pairing vanishes on $I\otimes H$ if, in addition to $\langle i,h\rangle = 0$ for all $i\in K$ and $h\in K_H$, one has that $\Delta^2(h) \subset H\otimes \mathrm{Span}_\genfd(K_H\cup \{1\})\otimes H$ for all $h\in K_H$.
\end{enumerate}
\end{lemma}

\begin{proof}
(i) Fix any $i\in I$. Then $\Delta(i) = \sum_\alpha i_\alpha\otimes b_\alpha+\sum_\beta b'_\beta\otimes i'_\beta$ for some $i_\alpha\in I,b_\beta\in B$. Then 
$\langle i, h h'\rangle = \sum_\alpha\langle i_\alpha,h\rangle\langle b_\alpha,h'\rangle + \sum_\beta\langle b'_\beta,h\rangle\langle i'_\beta,h'\rangle = 0$.

(ii) To show $\langle b i b', h\rangle = 0$ one needs $\langle b,h_{(1)}\rangle \langle i,h_{(2)}\rangle\langle b',h_{(3)}\rangle = 0$.  The condition on $\Delta^2(h)$ is sufficient for this to hold for all $i\in I$, $b,b'\in B$, $h\in K_H$.
\end{proof}
One often starts by constructing an auxiliary pairing where one of the bialgebras is free~\cite{radford}. If $C$ is a coalgebra, then by the universal property of the tensor algebra $T(C)$ there is a unique algebra map $T(C)\to T(C)\otimes T(C)$ extending the composition $C\stackrel{\Delta}\rightarrow C\otimes C\hookrightarrow T(C)\otimes T(C)$ along inclusion $C\hookrightarrow T(C)$; this is a comultiplication on $T(C)$ making it into a bialgebra. Every coalgebra map $C\to B$ to a bialgebra admits a unique extension to a bialgebra map $T(C)\to B$.

We need a different variant of this standard universal property. Suppose $C=V\otimes\genfd u$ as a vector space where $u$ is \emph{grouplike}, that is, $\Delta(u)=u\otimes u$,  $\epsilon(u)=1$. Ideal $I_u$ in $T(C)$ generated by $u-1$ is a biideal because $\Delta(u-1) = u\otimes (u-1)+(u-1)\otimes 1$ and $\epsilon(u-1)=0$. Composition $T(V)\hookrightarrow T(V\oplus\genfd u)\to T(V\oplus\genfd u)/I_u$ is an isomorphism of algebras (the inverse can easily be described); we transfer the comultiplication from $T(V\otimes\genfd u)/I_u$ to $T(V)$ along this isomorphism. By the above universal property, any coalgebra map $f\colon V\oplus\genfd u\to B$, where $B$ is a bialgebra, extends uniquely to a bialgebra map $T(V\oplus\genfd u)\to B$; if $f(u) = 1$ then it induces a bialgebra map $T(V\oplus\genfd u)/I_u\to B$. 

\begin{lemma}\label{lem:freebialgvar}
Suppose $V\in\Vect$ and $C = (V\oplus\genfd,\Delta,\epsilon)$ is a coalgebra such that $0\oplus 1$ is grouplike. Then $T(V)$ has a canonical bialgebra structure such that the inclusion $V\oplus\genfd\hookrightarrow T(V)$ is a coalgebra map and for any bialgebra $B$, each coalgebra map $C\to B$ mapping $0\oplus 1$ to $1_B$ admits a unique extension to a bialgebra map $T(V)\to B$.
\end{lemma}

\subsection{Yetter--Drinfeld module algebras}
\label{ss:YD}
In this subsection, fix a Hopf $\genfd$-algebra $H = (H,\Delta,\epsilon)$ with comultiplication $\Delta\colon h\mapsto \sum h_{(1)}\otimes h_{(2)}$ and counit $\epsilon\colon H\to\genfd$. Recall that the category $\mathcal{M}_H$ of right $H$-modules is monoidal: if $(M,\blacktriangleleft_M)$ and $(N,\blacktriangleleft_N)$ are $H$-modules then their tensor product is $\genfd$-module $M\otimes_\genfd N$ with $H$-action $\blacktriangleleft\colon(m\otimes n)\otimes h\mapsto (m\blacktriangleleft_M h_{(1)})\otimes(n\blacktriangleleft_N h_{(2)})$ and the unit object is $\genfd$ with action $c\blacktriangleleft h =\epsilon(h) c$, for $m\in M$, $n\in N$, $h\in H$, $c\in\genfd$.
A right $H$-module algebra $A$ is a monoid in $\mathcal{M}_H$: a right $H$-module $(A,\blacktriangleleft\colon A\otimes H\to A)$ with multiplication $\cdot$ such that $\sum (a\blacktriangleleft h_{(1)})\cdot(b\blacktriangleleft h_{(2)}) = (a\cdot b)\blacktriangleleft h$ and $1\blacktriangleleft h = \epsilon(h) 1$. One can then form a \emph{smash product algebra} $H\sharp A$ with underlying $\genfd$-vector space $H\otimes A$ and associative multiplication $\cdot$ given by $(h\sharp a)\cdot (k\sharp b) := h k_{(1)}\sharp (a\blacktriangleleft k_{(2)}) b$ where $h\sharp a$ is an alias for $h\otimes a\in H\sharp A$. We often identify $a\in A$ with $a\sharp 1$ and $h\in H$ with $1\sharp h\in A\sharp H$ (thus for $a,b\in A$, $h,k\in H$, $a\cdot (h\sharp b)$ denotes $(1\sharp a)\cdot(h\sharp b)$ and $h\cdot (k\sharp b) = (h k)\sharp b$). We extend $\blacktriangleleft$ to a right action, also denoted $\blacktriangleleft$, of $H\sharp A$ on $A$ by setting $a\blacktriangleleft (h\sharp b) := (a\blacktriangleleft h) b\in A$.

A right-left \emph{Yetter--Drinfeld $H$-module} $(M,\blacktriangleleft,\lambda)$ is a unital right $H$-module $(M,\blacktriangleleft)$ with
a left $H$-coaction $\lambda\colon M\to H\otimes M$,
$m\mapsto \lambda(m) = \sum m_{[-1]}\otimes m_{[0]}$,
satisfying Yetter--Drinfeld compatibility condition
\begin{equation}\label{eq:YDorig}
f_{(2)} (m\blacktriangleleft f_{(1)})_{[-1]}  \otimes (m\blacktriangleleft f_{(1)})_{[0]} = m_{[-1]}f_{(1)}\otimes (m_{[0]}\blacktriangleleft f_{(2)}),\quad \text{ for all } m\in M, f\in H.
\end{equation}
Morphisms of Yetter--Drinfeld modules are morphisms of underlying modules which are also morphisms of comodules and the tensor product of Yetter--Drinfeld modules is the tensor product of the underlying $H$-modules equipped with the coaction $m\otimes n\mapsto n_{[-1]} m_{[-1]}\otimes m_{[0]}\otimes n_{[0]}$ (notice the order!). Thus we obtain a braided monoidal category ${}^H{\mathcal{YD}}_H$ of (right-left) Yetter--Drinfeld $H$-modules with braiding $\sigma_{M,N}\colon M\otimes N\to N\otimes M$ given by
$m\otimes n\mapsto (n\blacktriangleleft_N m_{[-1]})\otimes m_{[0]}$. For finite-dimensional $H$, ${}^H{\mathcal{YD}}_H$ is braided monoidally equivalent to the Drinfeld--Majid center of the monoidal category ${\mathcal{M}}_H$ of right $H$-modules.

If $A$ is a right $H$-module algebra with a left $H$-coaction $\lambda$ 
and if we identify the underlying vector spaces of $H\sharp A$ and $H\otimes A$,
then the Yetter--Drinfeld compatibility
may be rewritten in terms of the multiplication in $H\sharp A$, as
\begin{equation}\label{eq:YDsharp}
  f_{(2)}\cdot\lambda(a\blacktriangleleft f_{(1)})
  = \lambda(a)\cdot f,\quad \text{ for all } a\in A , f\in H.
\end{equation}

Monoids in ${}^H{\mathcal{YD}}_H$ are called (right-left) \emph{Yetter--Drinfeld
$H$-module algebras}. They are Yetter--Drinfeld modules with multiplication such that they become $H$-module algebras and $H^{\mathrm{op}}$-comodule algebras. Notice that an $H$-comodule is the same thing as an $H^{\mathrm{op}}$-comodule, but saying that it is a comodule algebra is different. If $(A,\blacktriangleleft,\lambda)$ is a Yetter--Drinfeld module and $\mu\colon A\otimes A\to A$ a $\genfd$-linear map, then $\mu$ is \emph{braided commutative} if $\sigma_{A,A}\circ\mu = \mu$, that is, $(a\blacktriangleleft b_{[-1]}) b_{[0]} = b a$ for all $a,b\in A$. An $H$-module algebra $(A,\blacktriangleleft)$ is \emph{braided commutative} if its multiplication is braided commutative.

\begin{lemma}\label{lem:}
  Consider an $H$-module algebra $(A,\blacktriangleleft)$ with multiplication $\mu$ and a coaction $\lambda$ so that $(A,\blacktriangleleft,\lambda)$ is a Yetter--Drinfeld $H$-module.
  \begin{enumerate}
	\item[(i)] Multiplication $\mu$ is braided commutative in ${}^H{\mathcal{YD}}_H$ if and only if for the extended action $\blacktriangleleft$ of the smash product $H\sharp A$  relation $a\blacktriangleleft\lambda(b) = b a$ holds.

  \item[(ii)] Multiplication $\mu$ is braided commutative if and only if
  all elements of the form $1\sharp a$, $a\in A$, commute with all elements of the form $\lambda(b)$, $b\in A$, viewed inside algebra $H\sharp A$.

  \item[(iii)] Suppose $\mu$ is braided commutative. Then $A$ is  an $H^{\mathrm{op}}$-comodule algebra (hence also a Yetter--Drinfeld module algebra) if and only if $\lambda$ considered as a map with values in smash product algebra $H\sharp A$ is antimultiplicative.
  \end{enumerate}
\end{lemma}
\begin{proof}
For (i) indeed, the left hand side is $a\blacktriangleleft (b_{[-1]}\sharp b_{[0]}) = (a\blacktriangleleft b_{[-1]}) b_{[0]}$. Parts (ii) and (iii) are left to the reader. They are implicit in~\cite{BrzMilitaru}.
\end{proof}
\subsection{Leibniz algebras}

Left and right Leibniz algebras are nonassociative algebras slightly generalizing Lie algebras by dropping the condition of antisymmetry.

A $\genfd$-vector space $\mathfrak{h}$ equipped with a linear map $[- , -] \colon \mathfrak{h} \otimes_\genfd \mathfrak{h} \to \mathfrak{h}$ is a \emph{left Leibniz algebra}~\cite{LodPir} if for every $x \in \mathfrak{h}$ the map  $\ad x\colon y  \mapsto [x,y]$ is a derivation on $\mathfrak{h}$, that is, if left Leibniz identity $[x,[y,z]] = [[x,y],z] + [y,[x,z]]$ holds for all $x,y,z \in \mathfrak{h}$. Let $\mathfrak{h}^l$ be a copy of vector space $\mathfrak{h}$, with elements denoted $l_x$,
$x\in\mathfrak{h}$, with operations transported via $x\mapsto l_x$.
Denote by $\mathfrak{h}_{{Lie}}$ the Lie algebra obtained as a quotient
of $\mathfrak{h}$ by two-sided ideal $I_{[x,x],x\in \mathfrak{h}}$ generated by all commutators $[x,x]$, $x\in\mathfrak{h}$. It is a Lie algebra and it is maximal in the sense that if $\mathrm{char}\,\genfd\neq 2$, every map $\mathfrak{h}\to\ggf$ to a Lie algebra $\ggf$ factors through $\mathfrak{h}_{{Lie}}$ (if $\mathrm{char}\,\genfd = 2$, relation $[x,x]=0$ is stronger than the antisymmetry).

\begin{lemma}\label{prop:uhlie} Let $\mathfrak{h}$ be a left Leibniz algebra. Universal enveloping algebra $U(\mathfrak{h}_{Lie})$ of Lie algebra $\mathfrak{h}_{Lie} \cong \mathfrak{h} / I_{[x,x],x\in \mathfrak{h}}$ is isomorphic to $T(\mathfrak{h}^l)/I^l$, where $I^l$ is the ideal in $T(\mathfrak{h}^l)$ generated by $l_{[x,y]} - l_x \otimes l_y + l_y \otimes l_x$, $x,y\in \mathfrak h$.
\end{lemma}

\begin{proof}
Left to the reader. Included in the preprint version, {\tt arXiv:2308.15467}.
\end{proof}

We say that $\genfd$-vector space $\mathfrak{h}$ together with a linear map $[- , - ] \colon \mathfrak{h} \otimes_\genfd \mathfrak{h} \to \mathfrak{h}$ is a \emph{right Leibniz algebra} if for every $x \in \mathfrak{h}$ the map $y \mapsto [y, x ]$ is a derivation on $\mathfrak{h}$, that is, if the right Leibniz identity $[[x,y],z] = [[x,z],y] + [x,[y,z]]$ holds for all $x,y,z \in \mathfrak{h}$. By quotienting $\mathfrak{h}$ by the ideal generated by $[x,x], x \in \mathfrak{h},$ we get a maximal quotient Lie algebra, $\mathfrak{h} \to \mathfrak{h}_{Lie}$.

\begin{lemma}\label{prop:uhlier} Let $\mathfrak{h}$ be a right Leibniz algebra. Universal enveloping algebra $U(\mathfrak{h}_{Lie})$ of Lie algebra $\mathfrak{h}_{Lie} \cong \mathfrak{h} / I_{[x,x],x\in \mathfrak{h}}$ is isomorphic to $T(\mathfrak{h}^r)/I^r$, where $I^r$ is the ideal in $T(\mathfrak{h}^r)$ generated by $r_{[x,y]} - r_x \otimes r_y + r_y \otimes r_x$, $x,y\in\mathfrak h$. 
\end{lemma}
%\begin{proof}
%	Analogous to proof of Lemma \ref{prop:uhlie}.
%\end{proof}

\section{$U(\mathfrak{h}_{{Lie}})$ as a Yetter--Drinfeld $\OO(\Aut(\mathfrak{h}))$-module algebra}
\label{sec:bcYD}

In this section, we prove the central result of this article: for any finite-dimensional Leibniz algebra $\mathfrak{h}$ over any field $\genfd$, the universal enveloping algebra $U(\mathfrak{h}_{{Lie}})$ of its maximal quotient Lie algebra $ \mathfrak{h}_{{Lie}}$ is a braided commutative Yetter--Drinfeld module algebra over the Hopf algebra $\OO(\Aut(\mathfrak{h}))$ of regular functions  on the algebraic group of automorphisms of~$\mathfrak{h}$. This result immediately implies that the smash product algebra $\OO(\Aut(\mathfrak{h}))\sharp U(\mathfrak{h}_{{Lie}})$ is a total algebra of a Hopf algebroid over $U(\mathfrak{h}_{Lie})^\op, U(\mathfrak{h}_{Lie})$, see Section~\ref{sec:Ha}.

\subsection{Hopf algebra $\OO(\Aut(L))$}\label{ss:Oaut}

Let $(L,\cdot_L)$ be any nonassociative algebra of finite dimension $n$ over a field $\genfd$. The general linear group of the underlying vector space, $\mathrm{GL}(L)$ is an affine algebraic group with algebra of regular functions $\OO(\mathrm{GL}(L))$ that is therefore a Hopf algebra via $\Delta(f)(M,N) = f(M\circ N)$ and $\epsilon(f) = f(1)$ for any $M,N\in\mathrm{GL}(L)$~\cite{Cartier}. For a chosen ordered basis $\mathbf{b} = (x_1,\ldots,x_n)$ of $L$, interpreting matrices as operators amounts to an isomorphism $\iota_{\mathbf{b}}\colon\mathrm{GL}(n,\genfd)\stackrel\cong\to\mathrm{GL}(L)$. Structure constants $C_{ij}^k = C_{\mathbf{b} ij}^k$ are defined by
\begin{equation}\label{eq:Lstrconst}
x_i\cdot x_j = \sum_{k=1}^n C_{\mathbf{b}ij}^k x_k,\quad i,j\in\{1,\ldots,n\},
\end{equation}
and we introduce as algebra generators of $\OO(\mathrm{GL}(n,\genfd))$ regular functions $U^i_j\colon M\mapsto M^i_j,\bar{U}^i_j\colon M\mapsto(M^{-1})^i_j$, where $M^i_j$ is the $(i,j)$-th entry of matrix $M\in\mathrm{GL}(n,\genfd)$. As an abstract algebra, $\OO(\mathrm{GL}(n,\genfd))$ is the free algebra on $n^2$ generators $U^i_j,\bar{U}^i_j$ modulo the $n^2$ relations $\sum_k U^i_k\bar{U}^k_j = \delta^i_j$. The comultiplication is then given by $\Delta(U^i_j) = \sum_k U^i_k\otimes U^k_j$ and $\Delta(\bar{U}^i_j) = \sum_k \bar{U}^k_j\otimes\bar{U}^i_k$ with counit $\epsilon(U^i_j) = \epsilon(\bar{U}^i_j) =\delta^i_j$. By definition, an element $\psi\in\mathrm{GL}(L)$ is an automorphism if $\psi(a)\cdot_L \psi(b) = \psi(a\cdot_L b)$ for all $a,b\in L$. These relations cut out the subgroup $\mathrm{Aut}(L)\subset\mathrm{GL}(L)$. To see that it is a Zariski closed subgroup, write $a = \sum_k a^k x_k$, $b = \sum_k b^k x_k$ and observe that this condition amounts to a system of $n^3$ polynomial equations in $\mathrm{GL}(n,\genfd)$,
$$
\sum_{r} C_{ij}^r \psi_r^k = \sum_{l,m} \psi_i^l \psi_j^m C_{lm}^k.
$$
In other words, $\iota_{\mathbf{b}}$ induces an identification $\iota_{\mathbf{b}}^*\colon\mathcal{O}(\Aut(L))\stackrel{\cong}\to\mathcal{O}(\Aut(L))_{\mathbf{b}}$ with the quotient $\OO(\Aut(L))_{\mathbf{b}}$ of $\OO(\mathrm{GL}(n,\genfd))$ by the ideal $I_{\Aut(L)\mathbf{b}}$ generated by relations
\begin{equation}\label{eq:CUU}
\sum_{l,m} C_{lm}^k U_i^l U_j^m  = \sum_{r}  U_r^k C_{ij}^r.
\end{equation}
Regarding that the inclusion of subvarieties $\mathrm{Aut}(L)\subset \mathrm{GL}(L)$ is also an inclusion of groups, this ideal is Hopf and $\mathcal{O}(\mathrm{Aut}(L))$ is the quotient Hopf algebra of functions on the subgroup. One can also directly check that the ideal $I_{\Aut(L)\mathbf{b}}$ is a Hopf ideal.

Denote by $\GG^i_j = \mathcal{G}^i_{\mathbf{b}j} = U^i_j + I_{\Aut(L)\mathbf{b}}$ and $\bar\GG^i_j = \bar{\mathcal{G}}^i_{\mathbf{b}j} = \bar{U}^i_j +I_{\Aut(L)\mathbf{b}}$ the generators of $\mathcal{O}(\mathrm{Aut}(L))_{\mathbf{b}}$.
If $T = (T^i_j)_{i,j=1}^n$ is a transition matrix to a basis $\mathbf{b}' = (x'_1,\ldots,x'_n)$, $x'_j = \sum_i T^i_j x_i$, then $\iota_{\mathbf{b}'}^{-1}\circ\iota_{\mathbf{b}}\colon\mathrm{GL}(n,\genfd)\to\mathrm{GL}(n,\genfd)$, $A \mapsto T A T^{-1}$.
Then $\GG^i_{\mathbf{b}'j}\mapsto\sum_{l,m} T^i_l\GG^l_{\mathbf{b}m} {T^{-1}}^m_j$ extends to a Hopf algebra isomorphism $\theta_{\mathbf{b}\mathbf{b}'}\colon\OO(\Aut(L))_{\mathbf{b}'}\to\OO(\Aut(L))_{\mathbf{b}}$ and $\theta_{\mathbf{b}'\mathbf{b}}=\theta_{\mathbf{b}\mathbf{b}'}^{-1}$.  This implies $\iota^{*-1}_{\mathbf{b'}}(\GG^i_{\mathbf{b}'j})= \sum_{l,m} T^i_l\iota^{*-1}_{\mathbf{b}}(\GG^l_{\mathbf{b}m}){T^{-1}}^m_j$ within $\OO(\Aut(\mathfrak{h}))$.
When it is clear which basis $\mathbf{b}$ is fixed, $\iota_{\mathbf{b}}^{*-1}(\GG^i_{\mathbf{b}j})\in\OO(\Aut(L))$ will also be denoted by $\GG^i_{\mathbf{b}j}$ or simply $\GG^i_j$. Assuming the identification $\iota_{\mathbf{b}}^{*-1}$, we write 
\begin{equation}\label{eq:G'TGTm}
\GG^i_{\mathbf{b}'j} = \sum_{l,m} T^i_l\GG^l_{\mathbf{b}m} {T^{-1}}^m_j.
\end{equation}
Assuming the identification $\iota_{\mathbf{b}}^{*-1}$, if $\psi$ is an automorphism of $L$, $\sum_i\GG^i_{\mathbf{b}j}(\psi)x_i = \psi(x_j)$.
%Use (\ref{eq:G'TGTm}).
The standard reasoning above is summarized in the following proposition.

\begin{proposition}\label{prop:OAut(L)} Let $L$ be a nonassociative algebra of finite dimension $n$ with a $\genfd$-basis $\mathbf{b}$ and structure constants $C^k_{ij} = C^k_{\mathbf{b}ij}$~(\ref{eq:Lstrconst}). 
Hopf algebra $\mathcal{O}(\Aut(L))$ of regular functions on the affine algebraic group of automorphisms of  $L$ is as an algebra isomorphic to a commutative algebra $\mathcal{O}(\Aut(L))_{\mathbf{b}}$ with $2 n^2$-generators $\mathcal{G}^i_j$, $\bar{\mathcal{G}}^i_j$, $i,j\in\{1,\ldots,n\}$ and defining relations
\begin{equation}
\sum_{l,m} C_{lm}^k \mathcal{G}_i^l \mathcal{G}_j^m  = \sum_{r}  \mathcal{G}_r^k C_{ij}^r,\quad \sum_k \mathcal{G}^i_k\bar{\mathcal{G}}^k_j = \delta^i_j = \sum_k\bar{\mathcal{G}}^i_k\mathcal{G}^k_j,\quad i,j,k\in\{1,\ldots,n\}.
\end{equation}
As a direct consequence the following
identities hold for all $i,j, k \in \{1,\ldots,n\}$:
\begin{equation}\label{eq:CGbarG}
\sum_{m,p} {\mathcal G}_p^i C_{mj}^p \bar{\mathcal G}_k^m  = \sum_m C_{km}^i {\mathcal G}_j^m ,
\end{equation}
\begin{equation}\label{eq:CbarGbarG}
\sum_{l,m} C_{lm}^k \bar{\mathcal{G}}_i^l \bar{\mathcal{G}}_j^m = \sum_{r} \bar{\mathcal{G}}_r^k C_{ij}^r.
\end{equation}
Isomorphism $\mathcal{O}(\Aut(L))_{\mathbf{b}}\cong \mathcal{O}(\Aut(L))$ is a Hopf algebra isomorphism if $\mathcal{O}(\Aut(L))_{\mathbf{b}}$ is given the unique comultiplication $\Delta$ and counit $\epsilon$ which are algebra maps satisfying
\begin{equation}\label{eq:OAutcopr}
 \Delta(\mathcal{G}^i_j) = \sum_k \mathcal{G}^i_k\otimes \mathcal{G}^k_j,\quad\Delta(\bar{\mathcal{G}}^i_j) = \sum_k \bar{\mathcal{G}}^k_j\otimes\bar{\mathcal{G}}^i_k,\quad\epsilon(\mathcal{G}^i_j) = \epsilon(\bar{\mathcal{G}}^i_j) =\delta^i_j,
\end{equation}
and the antipode $S$ satisfying $S(\bar{\mathcal{G}}^i_j) = \mathcal{G}^i_j$,
$S(\mathcal{G}^i_j) = \bar{\mathcal{G}}^i_j$ for all $i,j\in\{1,\ldots,n\}$.
\end{proposition}

\subsection{Hopf pairing}

If $(B,\Delta,\epsilon)$ is a $\genfd$-bialgebra, then a \emph{differentiation} of $B$ is any $\genfd$-linear map $D\colon B\to\genfd$ such that Leibniz rule $D(b c) = D(b)\epsilon(c) + \epsilon(b)D(c)$ holds. In other words, it is a $\genfd_\epsilon$-valued derivation of $B$, where $\genfd_\epsilon$ is $\genfd$ with the trivial $B$-(bi)module structure coming from the counit. In Hopf algebraic language, a differentiation is a primitive element in the restricted dual bialgebra $B^\circ$. The following lemma is standard and elementary.

\begin{lemma}\label{lem:bialgdiff}
Let $B$ be any bialgebra such that its underlying algebra is the free unital commutative algebra with a set of free generators $F_B$. There is a canonical isomorphism between the vector space $\genfd^{F_B}$ of set maps $F_B\to\genfd$ and the space of differentiations of $B$ which extend these maps.
\end{lemma}

Assume $V\in\Vect$ and $C=V\oplus\genfd$ is a coalgebra such that $\Delta(1)=1\otimes 1$ and $\Delta(v) = 1\otimes v+v\otimes 1$ for all $v\in V$. Suppose $V$ is paired with $B$ such that map $v\mapsto\langle v,-\rangle$ corestricts to a coalgebra map $\phi_1\colon C\to B^\circ$ for which $\phi_1(1) = 1_{B^\circ}=\epsilon_B$. Then $\phi_1(v)$ is a differentiation of $B$. Conversely, by Lemma~\ref{lem:bialgdiff} each such $\phi_1(v)$ is determined by $\phi_1(v)|_{F_B}$, where the values for the latter can be chosen independently. 

\begin{proposition}\label{prop:hpair}
	Let $\mathfrak{h}$ be a left Leibniz $\genfd$-algebra with a vector space basis $\mathbf{b} = (x_1, \ldots, x_n)$ and structure constants $C^i_{jk}$ determined from $[x_j,x_k] = \sum_i C^i_{jk} x_i$, $j,k\in\{1, \ldots, n\}$. In the notation of Subsection~\ref{ss:Oaut}, $\mathcal{G}^i_j, \bar{\mathcal{G}}^i_j$, $i,j\in\{1,\ldots,n\}$ are the generators of the algebra $\OO(\Aut(\mathfrak{h}))$. Denote also by $\tilde x$ the image of $x\in \mathfrak{h}$ in $\mathfrak{h}_{Lie}$.

Then there is a well defined and unique Hopf pairing 
	$$\langle-,-\rangle\colon U(\mathfrak{h}_{Lie}) \otimes \OO(\Aut(\mathfrak{h}))\to\genfd$$ such that $\langle \tilde{x}_k, \GG^i_j \rangle = C^i_{kj}$ for all $i,j,k \in \{1,\ldots,n \}$. 
This Hopf pairing does not depend on the choice of basis. 	
\end{proposition}
\begin{proof}
Notice that $C^i_{k j} = (\ad{x_k})^i_j$, the $(i,j)$-th matrix element  of $\ad{x_k}$, hence
$\langle \tilde{x},\GG^i_j\rangle = (\ad x)^i_{j}$ for all $\tilde{x}\in\mathfrak{h}$. 

\emph{Uniqueness.} If such a pairing exists, then, $\forall x\in\mathfrak{h}$, $0 = \epsilon_{U(\mathfrak{h}_{{Lie}})}(\tilde{x}) = \langle \tilde{x}, 1\rangle$, hence 
\begin{align*}
0 
&= \langle\tilde{x},\delta^i_j 1\rangle = \langle\tilde{x},\sum_m \mathcal{G}^i_m \bar{\mathcal{G}}^m_j\rangle = \langle\Delta_{U(\mathfrak{h}_{\mathrm{Lie}})}(\tilde{x}),\sum_m\mathcal{G}^i_m\otimes\bar{\mathcal{G}}^m_j\rangle \\
& =\sum_m\langle\tilde{x},\mathcal{G}^i_m\rangle\langle 1,\bar{\mathcal{G}}^m_j\rangle
+ \langle 1,\mathcal{G}^i_m\rangle\langle\tilde{x},\bar{\mathcal{G}}^m_j\rangle
= \sum_m (\ad\tilde{x})^i_m \epsilon(\bar{\mathcal{G}}^m_j) +
\epsilon(\mathcal{G}^i_m)\langle\tilde{x},\bar{\mathcal{G}}^i_m\rangle.
\end{align*} 
Thus,
$\langle\tilde{x},\bar{\mathcal{G}}^i_j\rangle= -\sum_m (\ad{\tilde{x}})^i_m\epsilon(\bar{\mathcal{G}}^m_j) = -(\ad{\tilde{x}})^i_j$ and, in particular, $\langle\tilde{x}_k,\bar{\mathcal{G}}^i_j\rangle = - C^i_{kj}$.

Denote $\OO^{1+} := \mathrm{Span}_\genfd\{\mathcal{G}^i_j\}_{i,j = 1}^n\subset \OO(\Aut(\mathfrak{h}))$. By~(\ref{eq:OAutcopr}),
$\Delta(\OO^{1+}) \in\OO^{1+}\otimes\OO^{1+}$, hence $\langle \tilde{x}_{k_1} \cdots \tilde{x}_{k_m} , \GG^i_j \rangle = \langle \tilde{x}_{k_1} \otimes \cdots \otimes \tilde{x}_{k_m} , \Delta^{m-1}(\GG^i_j) \rangle$ is a polynomial in expressions of the form $\langle \tilde{x}_{k},\GG^r_s\rangle$. Similarly, $\langle \tilde{x}_{k_1} \cdots \tilde{x}_{k_m} , \bar{\GG}^i_j \rangle$ are determined by $\langle \tilde{x}_{k},\bar\GG^r_s\rangle$.

For any $v\in U(\mathfrak{h}_{Lie})$, $\langle v,\tilde\GG^{i_1}_{j_1}\cdots \tilde\GG^{i_m}_{j_m}  \rangle = \langle  \Delta^{m-1}(v), \tilde\GG^{i_1}_{j_1}\otimes \cdots \otimes \tilde\GG^{i_m}_{j_m} \rangle$,
where each $\tilde\GG^{i_p}_{j_p}$ stands for either $\GG^{i_p}_{j_p}$ or $\bar\GG^{i_p}_{j_p}$. After expanding $\Delta^{m-1}(v)$, the right-hand side is written in terms of expressions of the form $\langle \tilde{x}_{k_1} \cdots \tilde{x}_{k_m} , \tilde{\GG}^i_j \rangle$. Therefore, if such a pairing exists, it is unique. 
	
	\emph{Existence.} We first consider the free commutative algebra $B_n$ on the $2n^2$ generators, still denoted $\mathcal{G}^i_j,\bar{\mathcal{G}}^i_j$, and with the same rule~(\ref{eq:OAutcopr}) for a bialgebra structure (this is the bialgebra of regular functions on the variety of pairs of arbitrary $n\times n$ matrices). A unique pairing of $\mathfrak{h}^l$ with $B$ is by Lemma~\ref{lem:bialgdiff} extending $\langle l_{x_k}, \GG^i_j \rangle = C^i_{kj}$, $\langle l_{x_k}, \bar\GG^i_j \rangle = - C^i_{kj}$ by Leibniz rule, requiring that $\langle l_{x_k}, -\rangle$ is a differentiation of $B_n$.
We now want to show that there is an induced pairing between $\mathfrak{h}^l$ and the quotient Hopf algebra $B_n/I = \mathcal{O}(\Aut(\mathfrak{h}))$; functionals $\langle l_{x_k}, -\rangle$ remain differentiations on the quotient. We need to show that the pairing restricted to $\mathfrak{h}^l\otimes I$ vanishes. The biideal of relations $I$ has a generating set $K_I$ of all elements of the form $\sum_{l,m} C_{lm}^k \mathcal{G}_i^l \mathcal{G}_j^m  - \mathcal{G}_r^k C_{ij}^r$, $\sum_k \mathcal{G}^i_k\bar{\mathcal{G}}^k_j -\delta^i_j$ or $\sum_k\bar{\mathcal{G}}^i_k\mathcal{G}^k_j - \delta^i_j$. Observe that $\epsilon(s) = 0$ for all $s\in K_I$. Thus for differentiation $D = \langle l_{x_k},-\rangle$ we obtain $D(b s) = D(b)\epsilon(s) + \epsilon(b)D(s) = 0$ for all $b\in B$ and $s\in K_I$. Therefore if the pairing vanishes on $K_I$ then it vanishes on the ideal generated by $K_I$.
	
Thus we need to check 
	$\langle l_{x_p} , \sum_{r,m} C_{rm}^k {\mathcal{G}}_i^r \mathcal{G}_j^m \rangle =\langle l_{x_p}, \sum_{n}  \mathcal{G}_n^k C_{ij}^n\rangle$
	and 
	$\langle l_{x_p}, \sum_j \mathcal{G}_j^i \bar{\mathcal{G}}_k^j \rangle = \langle l_{x_p}, \delta_k ^i \rangle = \langle l_{x_p} , \sum_j \bar{\mathcal{G}}_j^i \mathcal{G}_k^j \rangle$ for all $i,j,k,p \in \{1,\ldots,n\}$.
	The first equation is
              $$\sum_{r} C^k_{rj} C^r_{pi} + \sum_{m}  C_{im}^k  C^m_{pj}=  \sum_n C^k_{pn}C^n_{ij},$$ which is left Leibniz identity $[[x_p,x_i],x_j] + [x_i,[x_p,x_j]] = [x_p,[x_i,x_j]]$ in terms of the structure constants. By using the differentiation rule, the second equation is simply $\langle l_{x_p}, {\mathcal{G}}_k^i\rangle + \langle l_{x_p}, \bar{\mathcal{G}}_k^i\rangle = 0$, which holds for generators by definition. Therefore, there is a well defined pairing $(\mathfrak{h}^l\oplus\genfd)\otimes\OO(\Aut(\mathfrak{h})) \to \genfd$ such that $\langle 1,f\rangle = \epsilon(f)$ for all $f\in \OO(\Aut(\mathfrak{h}))$ and $\langle l_{x},-\rangle$ is a differentiation of $\OO(\Aut(\mathfrak{h}))$ for all ${x} \in \mathfrak{h}$, that is, 
	\begin{equation} \label{eq:lder}
	\langle l_{x}, fg\rangle = \langle l_{x}, f\rangle \epsilon(g) + \epsilon(f) \langle l_{x}, g\rangle,\quad\forall x\in \mathfrak h,\forall f,g \in \OO(\Aut(\mathfrak{h})).
	\end{equation}
This means that $\mathfrak{h}^l\oplus\genfd$ is equipped with a comultiplication $\Delta$ such that $\Delta(l_x) = 1\otimes l_x+l_x\otimes 1$ and $\epsilon(l_x)=0$ and the pairing respects $\Delta$: in the notation of Subsection~\ref{ss:pairing}, $c\mapsto \langle c,-\rangle$ restricts to a coalgebra map $\phi_1\colon \mathfrak{h}^l\oplus\genfd\to \OO(\Aut(\mathfrak{h}))^\circ$ sending $1$ to $1_{\OO(\Aut(\mathfrak{h}))^\circ}=\epsilon_{\OO(\Aut(\mathfrak{h}))}$. By Lemma~\ref{lem:freebialgvar} and the equivalence between Hopf pairings and bialgebra maps $T(\mathfrak{h}^l)\to \OO(\Aut(\mathfrak{h}))^\circ$, we extend this pairing to a unique Hopf pairing $T(\mathfrak{h}^l) \otimes \OO(\Aut(\mathfrak{h}))\to \genfd$; it is determined by the formula
	$$
        \langle l_{x_{i_1}} \cdots  l_{x_{i_m}} , f \rangle = \langle l_{x_{i_1}} \otimes \cdots \otimes l_{x_{i_m}} , \Delta^{m-1}(f) \rangle.
	$$ 	
	Denote by $I^l$ the ideal in $T(\mathfrak{h}^l)$ generated by $e_{x,y}:= l_{[x,y]} - l_x \otimes l_y + l_y \otimes l_x$, $x,y\in\mathfrak{h}$. By Lemma \ref{prop:uhlie}, $U(\mathfrak{h}_{Lie}) \cong T(\mathfrak{h}^l)/I^l$. Moreover, $\Delta_{T(\mathfrak{h}^l)}(e_{x,y}) = 1\otimes e_{x,y}+e_{x,y}\otimes 1$ and $\epsilon(e_{x,y})=0$, hence $I^l$ is a biideal. Clearly, $U(\mathfrak{h}_{Lie}) \cong T(\mathfrak{h}^l)/I^l$ as Hopf algebras as well.

We now check that the ideal generators $e_{x,y}$ of $I^l$ are paired with every element of $\OO(\Aut(\mathfrak{h}))$ as $0$. Relation $\langle l_{[x_k,x_n]} + l_{x_n} \otimes l_{x_k} , \GG^i_j\rangle = \langle l_{x_k} \otimes l_{x_n}  , \GG^i_j\rangle$ is equivalent to $\sum_m C_{kn}^m C^i_{mj} + \sum_p C^i_{np} C^p_{kj} = \sum_p C^i_{kp}C^p_{nj}, $ which restates the left Leibniz identity $[[x_k,x_n],x_j] + [x_n,[x_k,x_j]] = [x_k,[x_n,x_j]] $. Similarly, $\langle l_{[x_k,x_n]} + l_{x_n} \otimes l_{x_k} , \bar\GG^i_j\rangle = \langle l_{x_k} \otimes l_{x_n}, \bar\GG^i_j\rangle$ computes to the same identity. Since $I$ is a biideal and $\OO^1 := \mathrm{Span}_\genfd\{\mathcal{G}^i_j,\bar{\mathcal{G}}^i_j\}_{i,j=1}^n$ satisfies $\Delta_{\OO(\Aut(\mathfrak{h}))}(\OO^1)\subset\OO^1\otimes\OO^1$, we can apply Lemma~\ref{lem:coidealK}, part (ii), for $K_H = \OO^1$ to conclude that the pairing vanishes on the entire $I\otimes\OO(\Aut(\mathfrak{h}))$. Therefore, there is a well defined Hopf pairing $U(\mathfrak{h}_{Lie}) \otimes \OO(\Aut(\mathfrak{h}))\to\genfd$ satisfying $\langle \tilde{x}_k, \GG^i_j \rangle = C^i_{kj}$.

To show that the pairing does not depend on the choice of basis $\mathbf{b}$, note that we started from a pairing $\langle -,-\rangle_{\mathbf{b}}$ defined on $\mathfrak{h}^l\otimes\mathrm{Span}_\genfd\{\GG_{\mathbf{b}j}^i,\bar\GG_{\mathbf{b}j}^i\}_{i,j=1}^n\subset\mathfrak{h}^l\otimes B_n$ by $\langle l_x,\GG^i_{\mathbf{b}j}\rangle_{\mathbf{b}} = \ad(x)_{\mathbf{b}j}^i$ and $\langle l_x,\bar\GG^i_{\mathbf{b}j}\rangle_{\mathbf{b}}=-\ad(x)^i_{\mathbf{b}j}$. For a base change by a numerical matrix $T$, $\langle l_x,\GG^i_{\mathbf{b}'j}\rangle_{\mathbf{b}'} =  (\ad(x)_{\mathbf{b}'})^i_j =\sum_{m,l} T^i_m (\ad(x)_{\mathbf{b}})_l^m {T^{-1}}^l_j = \sum_{m,l} T^i_m \langle l_x,\GG_{\mathbf{b}l}^m\rangle_{\mathbf{b}}{T^{-1}}^l_j = \langle l_x,(T\mathcal{G}_{\mathbf{b}}T^{-1})^i_j\rangle_{\mathbf{b}} = \langle l_x, \theta_{\mathbf{b}\mathbf{b}'}(\mathcal{G}_{\mathbf{b'}j}^i)\rangle_{\mathbf{b}}$. Likewise for $\bar\GG$. 
Induced pairing $U(\mathfrak{h}_{Lie})\otimes\OO(\Aut(\mathfrak{h}))\to \genfd$ is uniquely defined by the pairing on $\mathfrak{h}^l\otimes\mathrm{Span}_\genfd\{\GG_{\mathbf{b}j}^i,\bar\GG_{\mathbf{b}j}^i\}_{i,j=1}^n$ and by abstract properties of the extension. Thus, it respects bialgebra isomorphism $\theta_{\mathbf{b}\mathbf{b}'}$ in the second argument. Once we quotient from $B_n$ down to $\OO(\Aut(\mathfrak{h}))$, $\theta_{\mathbf{b}\mathbf{b}'}$ becomes an identification $\iota_{\mathbf{b}}^{-1*}\circ\iota_{\mathbf{b}'}^*$ (extending (\ref{eq:G'TGTm})), yielding the invariance. 
\end{proof}

\begin{proposition}\label{prop:hpairr}
	Let $\mathfrak{h}$ be a right Leibniz $\genfd$-algebra and $\mathbf{b} = (y_1, \ldots, y_n)$ a $\genfd$-basis of~$\mathfrak{h}$. Denote by $C^i_{jk}$ structure constants determined from $[y_j,y_k] = C^i_{jk} y_i$, for $j,k\in\{1, \ldots, n\}$. Let $\mathcal{G}^i_j, \bar{\mathcal{G}}^i_j, \ i,j\in\{1,\ldots,n\}$ be the generators of the algebra $\OO(\Aut(\mathfrak{h}))$ from Subsection~\ref{ss:Oaut}. Denote by $\tilde{y}$ the image of $y\in \mathfrak{h}$ in $\mathfrak{h}_{Lie}$.

Then there is a well defined and unique Hopf pairing 
	$$\langle-,-\rangle\colon U(\mathfrak{h}_{Lie}) \otimes \OO(\Aut(\mathfrak{h}))\to\genfd$$ such that $\langle \tilde{y}_k, \GG^i_j \rangle = - C^i_{jk}$ for all $i,j,k\in\{1,\ldots,n\}$. 
This Hopf pairing does not depend on the choice of basis.
\end{proposition}

\begin{proof} Notice that $\langle \tilde{y}_k, \GG^i_j\rangle = -(\ad_r y_k)^i_j$ where $\ad_r y\colon z\mapsto [z,y]$ is the {\em right} adjoint action; thus the main difference from Proposition~\ref{prop:hpair} is change of side.

The entire proof is analogous to the proof of Proposition~\ref{prop:hpair}, hence we skip it. One first observes that, if the pairing exists, $\langle \tilde{y}_k, \bar\GG^i_j \rangle_{\mathbf{b}} = C^i_{\mathbf{b}jk}$ must hold. We are presenting $U(\mathfrak{h}_{Lie})$ as $T(\mathfrak{h}^r)/I^r$ from Lemma~\ref{prop:uhlier}. Key calculations with elements $l_{x}, x\in \mathfrak{h},$ which in Proposition~\ref{prop:hpair} boil down to the left Leibniz identity are now replaced by calculations with elements $r_y, y\in \mathfrak{h},$ (from Lemma~\ref{prop:uhlier})  and boil down to the right Leibniz identity. For example, $\langle r_{y_p} , \sum_{s,m} C_{sm}^k {\mathcal{G}}_i^s \mathcal{G}_j^m \rangle =\langle r_{y_p}, \sum_{n}  \mathcal{G}_n^k C_{ij}^n\rangle$ is $$\sum_{s} C^k_{sj} C^s_{ip} + \sum_{m}  C_{im}^k  C^m_{jp}=  \sum_n C^k_{np}C^n_{ij},$$ which is right Leibniz identity $[[y_i,y_p],y_j] + [y_i,[y_j,y_p]] = [[y_i,y_j], y_p]$. \end{proof}

\begin{remark}(Geometric origin of the pairing.) If $\genfd$ is $\mathbb{R}$ or $\mathbb{C}$ and $\mathfrak{h}$ is a Lie algebra $\mathfrak{g}$ over $\genfd$, then $\Aut(\ggf)$ is a linear Lie group and its Lie algebra is $\operatorname{Der}(\ggf)$.
Differential $df_\id$ of function $f \in \OO(\Aut(\ggf))$ at the unit $\id$ of $\Aut(\ggf)$ is a linear functional on $T_\id(\Aut(\ggf)) \cong \operatorname{Der}(\ggf)$, and therefore  $df_\id \in \operatorname{Der}(\ggf)^*$. Let $\ad X \colon \ggf \to \ggf$, $\ad X \colon Z \mapsto [X,Z]$. Then $\ad X \in \operatorname{Der}(\ggf)$. % for $X \in \ggf$. 
	
	We prove that the pairing $\Ug \otimes \OO(\Aut(\ggf)) \to \genfd$ from Proposition~\ref{prop:hpair},
in the case when $\genfd$ is $\mathbb{R}$ or $\mathbb{C}$ and $\mathfrak{h}$ is a Lie algebra $\mathfrak{g}$, agrees on subset $\ggf\otimes\OO(\Aut(\ggf))$ of its domain with the pairing $\langle -,- \rangle'\colon\ggf \otimes \OO(\Aut(\ggf))\to\genfd$ defined by 
	$$\langle X, f \rangle' = df_\id(\ad X ), \quad \text{ for } X \in \ggf \text{ and } f \in \OO(\Aut(\ggf)).$$  
	First we check that indeed $d(\mathcal{G}^i_j)_\id (\ad X_k) = C^i_{kj}.$ The exponential map $\operatorname{exp}$ maps a neighborhood of $0$ in $\operatorname{Der}(\ggf)$ to a neighborhood of $\id$ in $\Aut(\ggf)$.	We have that
	\begin{align*}
	d(\mathcal{G}^i_j)_\id (\ad X_k) & = (\ad X_k) (\mathcal{G}^i_j )(\id) = \lim_{t\to 0}\frac{d}{dt}\, \mathcal{G}^i_j(\exp (t \ad X_k) ) \\ 
	&= \lim_{t\to 0}\frac{d}{dt}\, \mathcal{G}^i_j\left(\sum_{r=0}^\infty \frac{(t \ad X_k)^r}{r!} \right) 
	= \lim_{t\to 0}\frac{d}{dt}\! \left(\sum_{r=0}^\infty \frac{(t \ad X_k)^r}{r!} \right)^i_j \\
	&= \lim_{t\to 0}\left(\sum_{r=1}^\infty \frac{t^{r-1} (\ad X_k)^r}{(r-1)!} \right)^i_j 
	=  (\ad X_k)^i_j = C^i_{kj}.
	\end{align*}
	Similarly, one checks that $d(\bar{\mathcal{G}}^i_j)_\id(\ad X_k) = - C^i_{kj}$, by using that $\exp(t \ad X_k)^{-1} = \exp(t \ad (-X_k))$. By linearity, we conclude that the pairings agree for all $X\in \ggf$ and generators $\mathcal{G}^i_j, \bar{\mathcal{G}}^i_j, i,j\in\{ 1, \ldots, n\}$. Since the pairing also has the property 
	$$\langle X, fg\rangle' = \langle X\otimes 1 + 1\otimes X, f\otimes g \rangle', \text{ for } X\in \ggf \text{ and } f,g\in \OO(\Aut(\ggf)),$$ we conclude that they agree for all $X \in \ggf$ and $f \in \OO(\Aut(\ggf))$.
\end{remark}

\subsection{Main theorem}

\begin{theorem}  \label{prop:Oaut} Let $\mathfrak{h}$ be a left Leibniz $\genfd$-algebra with vector space basis $\mathbf{b}= (x_1, \ldots, x_n)$ and structure constants $C_{ij}^k$
determined from $[x_i, x_j] = \sum_k C_{ij}^k x_k$, $i,j \in \{1,\ldots,n\}.$ Let $\mathcal{G}^i_j, \bar{\mathcal{G}}^i_j, \ i,j\in\{1,\ldots,n\}$ be the generators of the algebra $\OO(\Aut(\mathfrak{h}))$ from Subsection~\ref{ss:Oaut}. Denote by $\tilde{x}$ the image of $x\in \mathfrak{h}$ in $\mathfrak{h}_{Lie}$. Then the following holds.
\begin{enumerate}
\item[(i)] Hopf pairing $U(\mathfrak{h}_{Lie}) \otimes \OO(\Aut(\mathfrak{h})) \to \genfd$ from Proposition~\ref{prop:hpair} induces a right Hopf action $\btl \colon U(\mathfrak{h}_{Lie}) \otimes \OO(\Aut(\mathfrak{h})) \to U(\mathfrak{h}_{Lie})$ by  formula
\begin{equation}\label{eq:ract}
\tilde x \btl f := \langle {\tilde x}_{(1)} , f \rangle {\tilde x}_{(2)}, \quad \text{ for } \tilde x \in U(\mathfrak{h}_{Lie}) \text{ and } f \in \OO(\Aut(\mathfrak{h})),
\end{equation}
	which further induces the structure of a smash product algebra $\OO(\Aut(\mathfrak{h})) \sharp U(\mathfrak{h}_{Lie})$. This action and the smash product do not depend on the choice of basis $\mathbf{b}$.
\item[(ii)] There is a unique $\genfd$-linear unital antimultiplicative map 
$$\lambda \colon U(\mathfrak{h}_{Lie}) \to \OO(\Aut(\mathfrak{h})) \sharp U(\mathfrak{h}_{Lie})$$ such that
\begin{equation}\label{eq:lambda}
\lambda (\tilde{x}_j) = \sum_i \bar{\mathcal{G}}_j^i \sharp \tilde{x}_i, \quad \text{ for }  j \in \{1,\ldots,n\}.
\end{equation}
Map $\lambda$ does not depend on the choice of basis $\mathbf{b}$.
\item[(iii)] Elements of $\operatorname{Im}\lambda $ commute with elements of $1\sharp U(\mathfrak{h}_{Lie})$ in $\OO(\Aut(\mathfrak{h}))\sharp U(\mathfrak{h}_{Lie})$.
\item[(iv)] Map $\lambda$ is a left $\OO(\Aut(\mathfrak{h}))$-coaction on $U(\mathfrak{h}_{Lie})$. 
\item[(v)] $(U(\mathfrak{h}_{Lie}),\btl,\lambda)$ is a braided commutative right-left Yetter--Drinfeld module algebra over  $\OO(\Aut(\mathfrak{h}))$.
\end{enumerate}
\end{theorem}
\begin{proof} (i) Every Hopf pairing induces a right Hopf action in this way. By Proposition~\ref{prop:hpair}, the pairing, hence also the action, does not depend on the choice of basis. 
	
	(ii) We prove that such $\lambda$ exists. We first define auxiliary map $\tilde\lambda$ as a linear map $\mathfrak{h}^l \to \OO(\Aut(\mathfrak{h})) \sharp U(\mathfrak{h}_{Lie})$ such that $\tilde\lambda (l_{x_j}) = \sum_i \bar{\mathcal{G}}_j^i \sharp \tilde{x}_i$ for $j \in \{1,\ldots,n\}$, then expand it to $\tilde \lambda \colon T(\mathfrak{h}^l) \to \OO(\Aut(\mathfrak{h})) \sharp U(\mathfrak{h}_{Lie})$ by antimultiplicativity and then we check that $\tilde\lambda(I^l) = \{0\}$, where $I^l$ is the ideal in $T(\mathfrak{h}^l)$ generated by $l_{[x,y]} - l_x\otimes l_y + l_y\otimes l_x,$ $x,y \in \mathfrak{h}$. We compute 
	\begin{align*} \tilde\lambda(l_{x_i} \otimes l_{x_j}) &= \tilde\lambda(l_{x_j}) \cdot \tilde\lambda(l_{x_i}) = (\sum_k \bar{\mathcal G}_j^k\sharp \tilde{x}_k) \cdot( \sum_m \bar{\mathcal G}_i^m \sharp \tilde{x}_m) = \sum_{k,m,p} \bar{\mathcal G}_j^k \bar{\mathcal G}^p_i (\tilde{x}_k \btl \bar{\mathcal G}_p^m) \tilde{x}_m \\
	& = \sum_{k,m,p} \bar{\mathcal G}_j^k \bar{\mathcal G}^p_i \sharp (\delta_p^m \tilde{x}_k  - C_{kp}^m) \tilde{x}_m   = \sum_{k,m} \bar{\mathcal G}_j^k \bar{\mathcal G}^m_i \sharp \tilde{x}_k \tilde{x}_m  - \sum_{k,m,p} \bar{\mathcal G}_j^k \bar{\mathcal G}^p_i  C_{kp}^m\sharp \tilde{x}_m \\
	&\stackrel{(\ref{eq:CbarGbarG})}= \sum_{k,m} \bar{\mathcal G}_j^k \bar{\mathcal G}^m_i \sharp \tilde{x}_k \tilde{x}_m  - \sum_{k,m} \bar{\mathcal G}_m^k C_{ji}^m\sharp \tilde{x}_k.
	\end{align*}
	Analogously, $\tilde\lambda(l_{x_j} \otimes l_{x_i})=\sum_{k,m} \bar{\mathcal G}_i^m \bar{\mathcal G}^k_j \sharp \tilde{x}_m \tilde{x}_k  - \sum_{k,m} \bar{\mathcal G}_k^m C_{ij}^k\sharp \tilde{x}_m$.
After subtracting,
	\begin{align*}
	\tilde\lambda(l_{x_i} \otimes l_{x_j} - l_{x_j} \otimes l_{x_i}) &= \sum_{k,m} \bar{\mathcal G}_i^m \bar{\mathcal G}^k_j \sharp [\tilde{x}_k , \tilde{x}_m] - \sum_{k,m} \bar{\mathcal G}_m^k C_{ji}^m\sharp \tilde{x}_k  +  \sum_{k,m} \bar{\mathcal G}_k^m C_{ij}^k\sharp \tilde{x}_m \\
	 &= \sum_{k,m} \bar{\mathcal G}_i^m \bar{\mathcal G}^k_j \sharp \widetilde{[x_k , x_m]} - \sum_{k,m} \bar{\mathcal G}_m^k C_{ji}^m\sharp \tilde{x}_k  +  \sum_{k,m} \bar{\mathcal G}_k^m C_{ij}^k\sharp \tilde{x}_m \\
	 &= \sum_{k,m} \bar{\mathcal G}_i^m \bar{\mathcal G}^k_j C_{km}^p \sharp \tilde{x}_p - \sum_{k,m} \bar{\mathcal G}_m^k C_{ji}^m\sharp \tilde{x}_k  +  \sum_{k,m} \bar{\mathcal G}_k^m C_{ij}^k\sharp \tilde{x}_m\\
	 &= \sum_{m} C_{ji}^m \bar{\mathcal G}_m^p \sharp \tilde{x}_p - \sum_{k,m} \bar{\mathcal G}_m^k C_{ji}^m\sharp \tilde{x}_k  +  \sum_{k,m} \bar{\mathcal G}_k^mC_{ij}^k\sharp \tilde{x}_m \\
	 & = \sum_{k,m} \bar{\mathcal G}_k^m C_{ij}^k\sharp \tilde{x}_m.
	\end{align*}
	On the other hand, 
	$\tilde\lambda(l_{[x_i,x_j]}) = \tilde\lambda(\sum_p C^p_{ij} l_{x_p}) = \sum_{p,m} C_{ij}^p \bar{\mathcal G}_{p}^m \sharp \tilde{x}_m.$
	Equality $\tilde\lambda(l_{x_i} \otimes l_{x_j} - l_{x_j} \otimes l_{x_i})  = \tilde\lambda(l_{[x_i,x_j]})$ is now proven.	
	Therefore, by quotienting the domain of $\tilde \lambda$ by ideal $I^l$, we induce a well defined map $\lambda \colon U(\mathfrak{h}_{Lie}) \to \OO(\Aut(\mathfrak{h})) \sharp U(\mathfrak{h}_{Lie}) $.
	Additionally, we note that clearly
	\begin{equation}\label{eq:antimult} \lambda( v z) = \lambda( z)\lambda( v), \quad \text{ for all } v, z \in U(\mathfrak{h}_{Lie}).
	\end{equation}
To see that $\lambda=\lambda_{\mathbf{b}}$ defined by $\lambda_{\mathbf{b}}(\tilde{x}_j) = \sum_i \bar{\mathcal{G}}_{\mathbf{b}j}^i \sharp \tilde{x}_i$~(\ref{eq:lambda}) does not depend on the basis $\mathbf{b}$, we compute 
$\lambda_{\mathbf{b}'}(\tilde{x'}_j)
= \sum_s\bar\GG^s_{\mathbf{b'}j}\sharp\tilde{x'}_s
= \sum_{s,i}\bar\GG^s_{\mathbf{b'}j}\sharp T^i_s\tilde{x}_i
= \sum_{i,j,r,s} T^i_l {T^{-1}}^l_s\bar\GG^s_{\mathbf{b}'r} T^r_j\sharp\tilde{x}_i
\stackrel{(\ref{eq:G'TGTm})}= \sum_{i,l,m,s} T^i_l \bar\GG^l_{\mathbf{b}j} \sharp \tilde{x}_i 
= \sum_s T^i_j\lambda_{\mathbf{b}}(\tilde{x}_i) = \lambda_{\mathbf{b}}(\sum_s T^i_j\tilde{x}_i) = \lambda_{\mathbf{b}}(\widetilde{x'}_j)$.
	
	(iii) First we check that $\lambda(\tilde x_j)$ and $\tilde x_k$ commute for all $j,k \in \{1,\ldots,n\}$.
        \begin{align*}
	\tilde x_k \cdot \lambda(\tilde x_j) &=  \tilde x_k \cdot \sum_i  \bar{\mathcal G}_j^i \sharp \tilde x_i = \sum_{i,m} \bar\GG^m_j \sharp(\tilde x_k \btl \bar{\mathcal G}_m^i ) \tilde x_i \\
	&=		\sum_{i,m} \bar\GG^m_j \sharp(\delta_m^i \tilde x_k + C_{mk}^i)\tilde x_i = \sum_{m} \bar\GG^m_j \sharp(\tilde x_k \tilde x_m + [\tilde x_m,\tilde x_k]) = \\
	&= 	\sum_{m} \bar\GG^m_j \sharp \tilde x_m \tilde x_k = \lambda(\tilde x_j) \cdot \tilde x_k.
	\end{align*} By using (\ref{eq:antimult}), it is easy to prove the claim inductively for all elements of $U(\mathfrak{h}_{Lie})$,
	\begin{equation}\label{eq:imacom}
	z \cdot \lambda( v) = \lambda( v) \cdot  z, \quad \text{ for all }  v, z \in U(\mathfrak{h}_{Lie}).
	\end{equation}  

	(iv) Coaction axiom $(\Delta\otimes\id)\circ\lambda = (\id\otimes\lambda)\circ\lambda$ on generators $\tilde{x}_j$, $j\in \{1,\ldots,n\}$, is apparent from definitions~(\ref{eq:OAutcopr}) and~(\ref{eq:lambda}). Both sides of it evaluate to $\sum_{k,i} \bar{\mathcal G}_j^k \otimes \bar{\mathcal G}_k^i \sharp \tilde x_i$.
       It is now sufficient to show that, if the coaction axiom is true for $ v, z \in U(\mathfrak{h}_{Lie})$, then it is true for the product $v  z \in U(\mathfrak{h}_{Lie})$. We compute
	\begin{align}\label{eq:coabc} \lambda( zv) &\stackrel{(\ref{eq:antimult})}{=} \lambda(v) \lambda( z) = \sum  v_{[-1]}\sharp  v_{[0]} \cdot  z_{[-1]}\sharp  z_{[0]} \stackrel{(\ref{eq:imacom})}{=} \sum v_{[-1]}  z_{[-1]} \sharp z_{[0]} v_{[0]}, 
	\end{align}
	from which it follows that, because $v$ and $z$ are assumed to satisfy the coaction axiom identity,
	\begin{align*}
	((\id \otimes \lambda)\circ \lambda)(v z) & = \sum v_{[-1]}  z_{[-1]} \otimes \lambda(  z_{[0]} v_{[0]}) \\
	& =  \sum v_{[-1]} z_{[-1]} \otimes v_{[0][-1]} z_{[0][-1]} \otimes z_{[0][0]} v_{[0][0]} \\
	& = \sum v_{[-1](1)} z_{[-1](1)} \otimes v_{[-1](2)}  z_{[-1](2)} \otimes  z_{[0]} v_{[0]} \\
	& = \sum (v_{[-1]}  z_{[-1]})_{(1)} \otimes (v_{[-1]} z_{[-1]})_{(2)}  \otimes  z_{[0]} v_{[0]}
	\end{align*}
	and, on the other hand,  
	\begin{align*}
	((\Delta \otimes \id) \circ \lambda)( z v) & = (\Delta \otimes \id)(\sum v_{[-1]} z_{[-1]} \sharp  z_{[0]} v_{[0]}) \\
	& = \sum (v_{[-1]} z_{[-1]})_{(1)} \otimes (v_{[-1]} z_{[-1]})_{(2)}  \otimes  z_{[0]} v_{[0]}.
	\end{align*}

	Counitality of $\lambda$ is checked first for generators, $((\epsilon \otimes \id) \circ \lambda)(\tilde x_j) = \sum_i\epsilon(\bar\GG^i_j) \tilde x_i = \tilde x_j$ for every $j\in \{1,\ldots,n\}$, and then easily proven inductively by using formula (\ref{eq:coabc}).	

	(v) First, we prove the Yetter--Drinfeld property:
	$$\sum f_{(2)} \cdot \lambda({v}\btl f_{(1)}) = \lambda({v}) \cdot f, \quad \text{ for all } {v}\in U(\mathfrak{h}_{Lie}) \text{ and } f \in \OO(\Aut(\mathfrak{h})).$$
        It is $\genfd$-linear both in ${v}$ and in $f$, hence it is sufficient to show it for ${v}$ and $f$ being words in generators, by induction on the length of a word.
For ${v}=\tilde{x}_k$ and $f=\mathcal{G}^i_j$, 
	\begin{align*} 
	\sum_m {\mathcal G}_j^m \lambda(\tilde x_k \btl {\mathcal G}^i_m) &= \sum_m {\mathcal G}_j^m \lambda(\delta_m^i \tilde x_k + C_{km}^i) = {\mathcal G}_j^i \lambda(\tilde x_k) + \sum_m C_{km}^i {\mathcal G}_j^m,  \\
	\lambda(\tilde x_k) \cdot {\mathcal G}_j^i & = \sum_m \bar{\mathcal G}_k^m \sharp \tilde x_m \cdot {\mathcal G}_j^i =  \sum_{m,p} \bar{\mathcal G}_k^m {\mathcal G}_p^i \sharp (\tilde x_m \btl {\mathcal G}_j^p) = \sum_{m,p} \bar{\mathcal G}_k^m {\mathcal G}_p^i \sharp (\delta_j^p \tilde x_m + C_{mj}^p)  \\ 
	& 
	= \sum_m \bar{\mathcal G}_k^m {\mathcal G}_j^i \sharp \tilde x_m + \sum_{m,p} \bar{\mathcal G}_k^m {\mathcal G}_p^i C_{mj}^p
         \\ &
         =  {\mathcal G}_j^i \lambda(\tilde x_k) + \sum_{m,p} {\mathcal G}_p^i C_{mj}^p \bar{\mathcal G}_k^m
         \stackrel{(\ref{eq:CGbarG})}= {\mathcal G}_j^i \lambda(\tilde x_k) + \sum_m C_{km}^i {\mathcal G}_j^m.
	\end{align*}
	The Yetter--Drinfeld property for generators ${v}=\tilde{x}_k$
        and $f = \bar{\mathcal G}_j^i$ is proven analogously.
		
	If the identity is true for some $ v, z\in U(\mathfrak{h}_{Lie})$ and any $f\in\OO(\Aut(\mathfrak{h}))$ of the form $\GG^i_j$, then it also holds for the product $ v z$ and all generators of $\OO(\Aut(\mathfrak{h}))$, because $\Delta(\GG^i_j)= \sum_m \GG^i_m\otimes \GG^m_j$
        and $\Delta(\bar{G}^i_j) = \sum_m\bar\GG^m_j\otimes\bar\GG^i_m$. Indeed, 
	\begin{align*}
	\sum f_{(2)} \lambda(( v z) \btl f_{(1)}) & = \sum f_{(3)} \lambda(( v \btl f_{(1)})( z \btl f_{(2)})) \stackrel{(\ref{eq:antimult})}= \sum f_{(3)} \lambda( z \btl f_{(2)})\lambda( v \btl f_{(1)}) \\
	& =\sum\lambda( z)\cdot f_{(2)} \lambda(v \btl f_{(1)}) = \lambda(z) \lambda( v) \cdot f \stackrel{(\ref{eq:antimult})}= \lambda( v z) \cdot f.
	\end{align*}
	Therefore, by induction, the identity is true for all $v \in U(\mathfrak{h}_{Lie})$ and $f$ being $\mathcal{G}^i_j$ or $\bar{\mathcal{G}}^i_j$.
        
	If the identity holds for some $f$ and $g$ in $\OO(\Aut(\mathfrak{h}))$ and all $ v\in U(\mathfrak{h}_{Lie})$, then it also holds for the product $fg \in \OO(\Aut(\mathfrak{h}))$ and all $ v\in U(\mathfrak{h}_{Lie})$, by
	\begin{align*}\sum (fg)_{(2)} \lambda( v \btl (fg)_{(1)}) 
	& = \sum \sum f_{(2)} g_{(2)} \lambda( v \btl (f_{(1)}g_{(1)})) 	\\
	&= \sum \sum f_{(2)} g_{(2)} \lambda(( v \btl f_{(1)})\btl g_{(1)}) \\
	& = \sum \sum f_{(2)}  \lambda( v \btl f_{(1)}) g  =\lambda( v) fg.
	\end{align*}
	We conclude inductively that the Yetter--Drinfeld property holds.	
	
	Next, the comodule algebra property is actually proven in (\ref{eq:coabc}), by using (\ref{eq:imacom}).
	
	Finally, let us prove the braided commutativity property: 
	$$ z \btl \lambda( v) =  v z, \quad \text{ for all } v, z\in U(\mathfrak{h}_{Lie}).$$ First we check this on generators. For any two $j,k \in \{1,\ldots,n\}$ we have
	$$\tilde{x}_k \btl \sum_i \bar{\mathcal G}_j^i \sharp \tilde{x}_i = \sum_i (\delta_j^i \tilde{x}_k \tilde{x}_i - C_{kj}^i \tilde{x}_i) = \tilde{x}_k \tilde{x}_j - [\tilde{x}_k,\tilde{x}_j] = \tilde{x}_j \tilde{x}_k.$$ 
	Next, we use induction on the length of the word acted on by $\lambda(\tilde{x}_j)$ on the right, for every $\tilde{x}_j, j\in\{1,\ldots,n \}$. The step of induction is
	$$( v z) \btl \sum_i \bar{\mathcal G}_j^i \sharp\tilde x_i = \sum_{i,m} ( v  \btl  \bar{\mathcal G}_j^m)( z \btl \bar{\mathcal G}_m^i )\tilde x_i = \sum_{i,m} ( v  \btl  \bar{\mathcal G}_j^m)\tilde x_m  z =  \tilde x_j v z,\quad
	\forall v, z \in U(\mathfrak{h}_{Lie}). $$
        At last, the step of induction on the length of the word on the right is
$$ w\btl \lambda( z v) = (w \btl \lambda( v)) \btl \lambda( z) = ( vw) \btl \lambda(z) = ( zv)w, \quad\forall w, z, v \in U(\mathfrak{h}_{Lie}).$$
\end{proof}

\begin{theorem}  \label{prop:Oautr} Let $\mathfrak{h}$ be a right Leibniz $\genfd$-algebra with $\genfd$-basis $\mathbf{b} =(y_1, \ldots, y_n)$ and structure constants $C_{ij}^k$ determined from $[y_i, y_j] = \sum_k C_{ij}^k y_k$, $i,j \in \{1,\ldots,n\}.$ Let $\mathcal{G}^i_j, \bar{\mathcal{G}}^i_j, \ i,j\in\{1,\ldots,n\}$ be the generators of the algebra $\OO(\Aut(\mathfrak{h}))$ from Subsection~\ref{ss:Oaut}. Denote by $\tilde{y}$ the image of $y\in \mathfrak{h}$ in $\mathfrak{h}_{Lie}$.
	
	Then the Hopf pairing $U(\mathfrak{h}_{Lie}) \otimes \OO(\Aut(\mathfrak{h})) \to \genfd$ defined in Proposition~\ref{prop:hpairr} induces a right Hopf action $\btl \colon U(\mathfrak{h}_{Lie}) \otimes \OO(\Aut(\mathfrak{h})) \to U(\mathfrak{h}_{Lie})$ by  formula
	$$\tilde y \btl f := \langle {\tilde y}_{(1)} , f \rangle {\tilde y}_{(2)}, \quad \text{ for } \tilde y \in U(\mathfrak{h}_{Lie}) \text{ and } f \in \OO(\Aut(\mathfrak{h})),$$
	which further induces the structure of a smash product algebra $\OO(\Aut(\mathfrak{h})) \sharp U(\mathfrak{h}_{Lie})$. 
	
	Then there also exists a unique $\genfd$-linear unital antimultiplicative map $\lambda \colon U(\mathfrak{h}_{Lie}) \to \OO(\Aut(\mathfrak{h})) \sharp U(\mathfrak{h}_{Lie})$ such that $$\lambda (\tilde y_j) = \sum_i {\mathcal{G}}_j^i \sharp \tilde y_i, \quad \text{ for }  j \in \{1,\ldots,n\}.$$
        This unique map $\lambda$ is a left coaction. 
	 
	Furthermore, $(U(\mathfrak{h}_{Lie}),\btl,\lambda)$ is a braided commutative right-left Yetter--Drinfeld $\OO(\Aut(\mathfrak{h}))$-module algebra.
        Maps $\blacktriangleleft$ and $\lambda$ do not depend on the choice of basis  $\mathbf{b}$.	
\end{theorem}
\begin{proof}
Analogous to the proof of Theorem \ref{prop:Oaut}. 	
\end{proof}
\section{Hopf algebroid from Yetter--Drinfeld module $U(\mathfrak{h}_{Lie})$}
\label{sec:Ha}

Given an (associative) algebra $A$ ('base algebra'), a left (associative) $A$-bialgebroid is given by a tuple $(\mathcal{K},\mu,\alpha,\beta,\Delta,\epsilon)$ where $(\mathcal{K},\mu)$ is an algebra ('total algebra'), $\alpha\colon A\to\mathcal{K}$ and $\beta\colon A^{\mathrm{op}}\to\mathcal{K}$ are algebra maps called source and target maps respectively which satisfy $\alpha(a)\beta(b)=\beta(b)\alpha(a)$ for all $a,b\in A$ hence equipping $\mathcal{K}$ with a structure of $A$-bimodule via $a.k.b = \alpha(a)\beta(b)k$ for $a,b\in A, h\in \mathcal{K}$ (and moreover of an $A\otimes A^{\mathrm{op}}$-ring). Comultiplication $\Delta\colon\mathcal{K}\to\mathcal{K}\otimes_A\mathcal{K}$ and counit $\epsilon\colon\mathcal{K}\to A$ are required to be $A$-bimodule maps which make ${}_A\mathcal{K}_A$ into a comonoid in the category of $A$-bimodules. Nontrivial compatibilities of the comonoid structure with $A\otimes A^{\mathrm{op}}$-ring structure on $\mathcal{K}$ are required~\cite{bohmHbk,BrzMilitaru,Lu} which radically simplify if the base algebra $A$ is commutative. A right $A$-bialgebroid is structure $(\mathcal{K},\mu,\alpha,\beta,\Delta,\epsilon)$~\cite{bohmHbk,Lu} such that  $(\mathcal{K},\mu,\beta,\alpha,\Delta^{\mathrm{op}},\epsilon)$ is a left $A^{\mathrm{op}}$-bialgebroid.
A Hopf $A$-algebroid~\cite{bohmHbk,Lu} should be an $A$-bialgebroid with an antihomomorphism of algebras $\tau\colon\mathcal{K}\to\mathcal{K}$ called antipode and with axioms generalizing that of antipode of a Hopf algebra. Commutative Hopf algebroids are a classical subject studied since 1960-s and appear as function algebras on groupoids. Several nonequivalent definitions of Hopf algebroids over a noncommutative base algebra appeared in 1990-s, including Lu--Hopf algebroids~\cite{Lu} which lack symmetries and involve a somewhat ad hoc section map. We consider symmetric Hopf algebroids~\cite{bohmHbk} where $\mathcal{K}$ has a structure of a left $A_L$-bialgebroid $(\mathcal{K},\mu,\alpha_L,\beta_L,\Delta_L,\epsilon_L)$ and a right $A_R$-bialgebroid $(\mathcal{K},\mu,\alpha_R,\beta_R,\Delta_R,\epsilon_R)$ with given isomorphism of algebras $A_L^{\mathrm{op}}\stackrel\cong\to A_R$ and antipode $\tau\colon\mathcal{K}\to\mathcal{K}$ satisfying a list of axioms~\cite{bohmHbk,halg}.

Given any Hopf algebra $H$ with a bijective antipode and a braided commutative Yetter--Drinfeld module algebra $A$ over $H$, smash product $H\sharp A$ is a symmetric Hopf algebroid with $A_L = A^\op$, $A_R = A$ by~\cite{bohmHbk,stojic}, adapting constructions from~\cite{BrzMilitaru,Lu}. This can be applied in the case $H=\OO(\Aut(\mathfrak{h}))$, $A=U(\mathfrak{h}_{Lie})$ from Section~\ref{sec:bcYD}. Hopf algebroid $\OO(\Aut(\mathfrak{h}))\sharp U(\mathfrak{h}_{Lie})$ is, in the case when $\mathfrak{h}$ is a Lie algebra, related to more geometric examples in~\cite{omin} and to the completed Hopf algebroid in~\cite{halg}. 

\subsection{Formulas for the Hopf algebroid $\mathcal{O}(\Aut(\mathfrak{h}))\sharp U(\mathfrak{h}_{Lie})$}
Here we write general formulas for the Hopf algebroid $\OO(\Aut(\mathfrak{h}))\sharp U(\mathfrak{h}_{Lie}) $ for a left Leibniz algebra $\mathfrak{h}$  and specify them also on generators. We use formulas for the right bialgebroid given in Proposition~4.1 of \cite{stojic}, and formulas for the left bialgebroid written in terms of the above smash product and formula for the antipode given in Corollary~4.1 of~\cite{stojic}. For the latter, one uses certain natural antiisomorphism $\phi \colon U(\mathfrak{h}_{Lie})^{\op} \to U(\mathfrak{h}_{Lie})$ or $\theta \colon U(\mathfrak{h}_{Lie}) \to U(\mathfrak{h}_{Lie})^{\op}$, as explained in \cite{stojic}, such that in the resulting Hopf algebroid $\phi = \epsilon_R \circ \alpha_L$ and $\theta = \epsilon_L\circ \alpha_R$. All above formulas are displayed in short in the table below Corollary~4.1 in \cite{stojic}. 
Formulas for a right Leibniz algebra can be derived similarly.

After that, we specify these formulas in the case of the Hopf algebroid corresponding to (in fact, embedded into) the Lie algebra noncommutative phase space of \cite{halg}. 

\subsubsection{For a left Leibniz algebra}
For a left Leibniz algebra $\mathfrak{h}$, in the notation of \cite{stojic} $H = \OO(\Aut(\mathfrak{h}))$, $R= U(\mathfrak{h}_{Lie})$ and antiisomorphism $\phi \colon U(\mathfrak{h}_{Lie})^{\op} \to U(\mathfrak{h}_{Lie})$ maps generators $\tilde x_{1\phi}^{\op}, \ldots, \tilde x_{n\phi}^{\op}$ of $U(\mathfrak{h}_{Lie})^{\op}$ by $\tilde x_{j\phi}^{\op} \mapsto \tilde x_j$, $j\in \{1,\ldots,n\}$, to generators $\tilde x_1,\ldots, \tilde x_n$ of $ U(\mathfrak{h}_{Lie})$. Since $(U(\mathfrak{h}_{Lie}), \btl, \lambda)$ is  a braided commutative right-left Yetter--Drinfeld module algebra over $ \OO(\Aut(\mathfrak{h}))$ by Theorem~\ref{prop:Oaut}, by Corollary~4.1 part (2) in \cite{stojic}, $H\sharp R = \OO(\Aut(\mathfrak{h}))\sharp U(\mathfrak{h}_{Lie})$ is a symmetric Hopf algebroid over $U(\mathfrak{h}_{Lie})^{\op},U(\mathfrak{h}_{Lie})$: it is (i)  a right bialgebroid over $U(\mathfrak{h}_{Lie})$ with structure maps 
\begin{align*}
&\alpha_R \colon U(\mathfrak{h}_{Lie}) \to H \sharp U(\mathfrak{h}_{Lie}), && \alpha_R(\tilde x) = 1_{H} \sharp \tilde x,
\\
&\beta_R \colon U(\mathfrak{h}_{Lie}) \to H\sharp U(\mathfrak{h}_{Lie}), && \beta_R(\tilde x) = \lambda(\tilde x), 
\\
& && \beta_R(\tilde x_j) = \sum_i \bar{\mathcal{G}}^i_j \sharp \tilde x_i, 
\\
&\Delta_R \colon H\sharp U(\mathfrak{h}_{Lie}) \to H\sharp U(\mathfrak{h}_{Lie}) \otimes_{U(\mathfrak{h}_{Lie})} H\sharp U(\mathfrak{h}_{Lie}), && \Delta_R(f\sharp \tilde x) = f_{(1)} \sharp 1_{U(\mathfrak{h}_{Lie})} \otimes_{U(\mathfrak{h}_{Lie})} f_{(2)} \sharp \tilde x ,
\\
&\epsilon_R \colon H\sharp U(\mathfrak{h}_{Lie}) \to  U(\mathfrak{h}_{Lie}), && \epsilon_R(f\sharp \tilde x) = \epsilon(f) \tilde x,
\end{align*} 
(ii) a left bialgebroid over $U(\mathfrak{h}_{Lie})^{\op}$ with structure maps 
\begin{align*}
&\alpha_L \colon U(\mathfrak{h}_{Lie})^{\op} \to H \sharp U(\mathfrak{h}_{Lie}), && \alpha_L(\tilde x^{\op}) = \lambda(\phi(\tilde x^{\op})) = \phi(\tilde x^{\op})_{[-1]} \sharp \phi(\tilde x^{\op})_{[0]},
\\
& && \alpha_L(\tilde x_{j\phi}^{\op}) = \lambda(\tilde x_j) = \sum_i \bar{\mathcal{G}}^i_j \sharp \tilde{x}_i,\\ 
&\beta_L \colon U(\mathfrak{h}_{Lie})^{\op} \to H\sharp U(\mathfrak{h}_{Lie}), && \beta_L(\tilde x^{\op}) = \phi(\tilde x^{\op})_{[0]} \btl S^{-1}(\phi(\tilde x^{\op})_{[-1]}), 
\\ 
& && \beta_L(\tilde x_{j\phi}^{\op}) = \sum_i \tilde x_i \btl \mathcal{G}^i_j = 1_H \sharp \tilde{x}_j + \sum_i C^i_{ij}\sharp 1_{U(\mathfrak{h}_{Lie})},  \\
&\Delta_L \colon H\sharp U(\mathfrak{h}_{Lie}) \to H\sharp U(\mathfrak{h}_{Lie}) \otimes_{U(\mathfrak{h}_{Lie})^{\op}} H\sharp U(\mathfrak{h}_{Lie}), && \Delta_L(f\sharp \tilde x) = f_{(1)} \sharp 1_{U(\mathfrak{h}_{Lie})} \otimes_{U(\mathfrak{h}_{Lie})^{\op}} f_{(2)} \sharp \tilde x, \\
&\epsilon_L \colon H\sharp U(\mathfrak{h}_{Lie}) \to  U(\mathfrak{h}_{Lie})^{\op}, && \epsilon_L(f\sharp \tilde x) = \phi^{-1}(\tilde{x}_{[0]} \btl S^2(\tilde{x}_{[-1]}) Sf),
\\
& && \epsilon_L(f\sharp \tilde x_j) = \phi^{-1}( \sum_i \tilde{x}_i \btl \bar{\mathcal{G}}^i_j Sf) \\
& && \qquad \quad \ \, \, = \epsilon(f)\tilde{x}_{j\phi}^{\op} - \big(\langle \tilde x_j, f \rangle + \epsilon(f) \sum_i  C^i_{ij}\big)1_{U(\mathfrak{h}_{Lie})^{\op}}, 
\\
& && \epsilon_L(1_H\sharp \tilde x_j) =  \tilde{x}_{j\phi}^{\op} - \sum_i C^i_{ij}1_{U(\mathfrak{h}_{Lie})^{\op}}, %, \ \epsilon_L(f \sharp 1_{U(\mathfrak{h}_{Lie})}) = \epsilon(f) 
\end{align*} 
with (iii) antipode 
\begin{align*}
&\mathcal \tau \colon H\sharp U(\mathfrak{h}_{Lie}) \to H\sharp U(\mathfrak{h}_{Lie}), && \tau(f\sharp \tilde x) = \tilde{x}_{[0]} \cdot S^2(\tilde{x}_{[-1]})Sf, 
\\
& && \tau(f \sharp \tilde{x}_j) = \sum_i \tilde{x}_i \cdot \bar{\mathcal{G}}^i_j Sf = \sum_i \bar{\mathcal{G}}^i_j S(f_{(2)})\sharp (\tilde x_i \btl S(f_{(1)})) - \sum_{i} C^i_{ij}  Sf\sharp 1_{U(\mathfrak{h}_{Lie})},
\\
& && \tau(1_H \sharp \tilde{x}_j) 
= \sum_i \bar{\mathcal{G}}^i_j \sharp \tilde x_i - \sum_{i} C^i_{ij}\sharp 1_{U(\mathfrak{h}_{Lie})}, 
\end{align*} 
where in formulas $\tilde x \in U(\mathfrak{h}_{Lie})$, $\tilde x^{\op} \in U(\mathfrak{h}_{Lie})^{\op}$, $f\in H$, and $j\in \{1, \ldots,n\}$ as index of generators $\tilde x_j$ of $U(\mathfrak{h}_{Lie})$. 

Alternatively, formulas for the left bialgebroid can be written in terms of the antiisomorphism $\theta \colon U(\mathfrak{h}_{Lie}) \to U(\mathfrak{h}_{Lie})^{\op}$. Denote by $\tilde x_{1\theta}^{\op}, \ldots, \tilde x_{n\theta}^{\op}$ the generators of $U(\mathfrak{h}_{Lie})^{\op}$ defined by $\tilde x_{j\theta}^{\op} := \theta(\tilde x_j)$, $j \in \{1,\ldots,n\}$. Note that the generators $\tilde{x}_{j\phi}^{\op}$ are generally different from the generators $\tilde{x}_{j\theta}^{\op}$. The formulas for the structure maps of (ii') the left bialgebroid over $U(\mathfrak{h}_{Lie})^{\op}$ written by using $\theta$ are:  
\begin{align*}
&\alpha_L \colon U(\mathfrak{h}_{Lie})^{\op} \to H \sharp U(\mathfrak{h}_{Lie}), && \alpha_L(\tilde x^{\op}) = \theta^{-1}(\tilde x^{\op})_{[0]} \cdot S^2(\theta^{-1}(\tilde x^{\op})_{[-1]}) ,
\\
& && \alpha_L(\tilde x_{j\theta}^{\op}) 
= \sum_i x_i \cdot \bar{\mathcal{G}}^i_j 
= \sum_i \bar{\mathcal{G}}^i_j \sharp \tilde{x}_i  - \sum_{i,k} C^i_{ik} \bar{\mathcal{G}}^k_j \sharp 1_{U(\mathfrak{h}_{Lie})} 
\\
& && \qquad \quad \ \, = \sum_i \bar{\mathcal{G}}^i_j \sharp \tilde{x}_i  - \sum_{i} C^i_{ij} \sharp 1_{U(\mathfrak{h}_{Lie})},\\ 
&\beta_L \colon U(\mathfrak{h}_{Lie})^{\op} \to H\sharp U(\mathfrak{h}_{Lie}), && \beta_L(\tilde x^{\op}) = 1_{H} \sharp \theta^{-1}(\tilde x^{\op}), 
\\ 
& && \beta_L(\tilde x_{j\theta}^{\op}) = 1_H\sharp \tilde x_j, \\ 
&\Delta_L \colon H\sharp U(\mathfrak{h}_{Lie}) \to H\sharp U(\mathfrak{h}_{Lie}) \otimes_{U(\mathfrak{h}_{Lie})^{\op}} H\sharp U(\mathfrak{h}_{Lie}), && \Delta_L(f\sharp \tilde x) = f_{(1)} \sharp 1_{U(\mathfrak{h}_{Lie})} \otimes_{U(\mathfrak{h}_{Lie})^{\op}} f_{(2)} \sharp \tilde x,
\\
&\epsilon_L \colon H\sharp U(\mathfrak{h}_{Lie}) \to  U(\mathfrak{h}_{Lie})^{\op}, && \epsilon_L(f\sharp \tilde x) = \theta(\tilde x \btl S^{-1}(f)), 
\\
& && \epsilon_L(1_H\sharp \tilde x_j) =  \tilde{x}_{j\theta}^{\op} ,
\\
& && \epsilon_L(f\sharp \tilde x_j) = \epsilon(f)\tilde{x}_{j\theta}^{\op} - \langle \tilde x_j, f \rangle 1_{U(\mathfrak{h}_{Lie})^{\op}}. 
\end{align*} 

\subsubsection{Lie algebra type noncommutative phase space}
For convenience, we here write formulas for the structure maps of the Hopf algebroid $\OO(\Aut(\mathfrak{g}^R)) \sharp U(\mathfrak{g}^R)$ over $U(\mathfrak{h}^L), U(\mathfrak{g}^R)$ that is inside an \emph{ad hoc} completed version $\hat S(\mathfrak{g}^*) \sharp U(\mathfrak{g}^R) \cong U(\mathfrak{g}^L)\sharp \hat S(\mathfrak{g}^*)$ of a Hopf algebroid that is the Lie algebra type noncommutative phase space in \cite{halg}.

In the setup of \cite{halg}, $\mathfrak{h} = \mathfrak{g}^R$, $H = \OO(\Aut(\mathfrak{g}^R))$, $R= U(\mathfrak{g}^R)$, $\hat y_1,\ldots,\hat y_n$ are generators of $\mathfrak{g}^R$, $\hat x_1,\ldots,\hat x_n$ are the corresponding generators of $\mathfrak{g}^L$, and $\theta \colon U(\mathfrak{g}^R) \to U(\mathfrak{g}^L)$ is the antiisomorphism $\theta(\hat y_j) = \hat x_j$, $j\in \{1,\ldots,n\}$. We use the formulas above by writing $\hat y$ instead of $\tilde x$, $\hat y_j$ instead of $\tilde x_j$,  $\hat x$ instead of $\tilde x^{\op}$,  $\hat x_j$ instead of $\tilde x_{j\theta}^{\op}$, $\mathcal{O}$ instead of $\bar{\mathcal{G}}$ and $\bar{\mathcal{O}}$ instead of ${\mathcal{G}}$, $-C^i_{jk}$ instead of $C^i_{jk}$, $\hat z_j$ instead of $\tilde x_{j\phi}^\op$, and $\mathcal S$ instead of $\tau$. %We omit writing $1_H$ and $1_{U(\mathfrak{g}^R)}$ in formulas here.

Now $H\sharp R = \OO(\Aut(\mathfrak{g}^R))\sharp U(\mathfrak{g}^R)$ is a Hopf algebroid over $U(\mathfrak{g}^L),U(\mathfrak{g}^R)$: it is (i)  a right bialgebroid over $U(\mathfrak{g}^R)$ with structure maps 
\begin{align*}
&\alpha_R \colon U(\mathfrak{g}^R) \to H \sharp U(\mathfrak{g}^R), && \alpha_R(\hat y) = 1_{H} \sharp \hat y,
\\
&\beta_R \colon U(\mathfrak{g}^R) \to H\sharp U(\mathfrak{g}^R), && \beta_R(\hat y) = \lambda(\hat y) = \hat{y}_{[-1]} \sharp \hat{y}_{[0]}, 
\\
& && \beta_R(\hat y_j) = \sum_i {\mathcal{O}}^i_j \sharp \hat y_i, 
\\
&\Delta_R \colon H\sharp U(\mathfrak{g}^R) \to H\sharp U(\mathfrak{g}^R) \otimes_{U(\mathfrak{g}^R)} H\sharp U(\mathfrak{g}^R), && \Delta_R(f\sharp \hat y) = f_{(1)} \sharp 1_{U(\mathfrak{g}^R)} \otimes_{U(\mathfrak{g}^R)} f_{(2)} \sharp \hat y ,
\\
&\epsilon_R \colon H\sharp U(\mathfrak{g}^R) \to  U(\mathfrak{g}^R), && \epsilon_R(f\sharp \hat y) = \epsilon(f) \hat y,
\end{align*} 
(ii) a left bialgebroid over $U(\mathfrak{h}_{Lie})^{\op}$ with structure maps
\begin{align*}
&\alpha_L \colon U(\mathfrak{g}^L) \to H \sharp U(\mathfrak{g}^R), && \alpha_L(\hat x) = \theta^{-1}(\hat x)_{[0]} \cdot S^2(\theta^{-1}(\hat x)_{[-1]}) ,
\\
& && \alpha_L(\hat x_j) 
= \sum_i \hat y_i \cdot {\mathcal{O}}^i_j 
= \sum_i {\mathcal{O}}^i_j \sharp \hat{y}_i  + \sum_{i} C^i_{ij},
%= \sum_i {\mathcal{O}}^i_j \sharp \hat{y}_i  + \sum_{i} C^i_{ij} \sharp 1_{U(\mathfrak{g}^R)},
\\ 
& && \alpha_L(\hat z_j) = \lambda(\hat y_j) = \sum_i \OO^i_j \sharp \hat y_i, 
\\
&\beta_L \colon U(\mathfrak{g}^L) \to H\sharp U(\mathfrak{g}^R), && \beta_L(\hat x) = 1_{H} \sharp \theta^{-1}(\hat x), 
\\ 
& && \beta_L(\hat x_j) = \hat y_j, 
\\ 
& && \beta_L(\hat z_j) = \sum_i \hat y_j \btl \bar \OO^i_j = \hat y_j - \sum_i C^i_{ij}, 
\end{align*}
\begin{align*}
\null &\Delta_L \colon H\sharp U(\mathfrak{g}^R) \to H\sharp U(\mathfrak{g}^R) \otimes_{U(\mathfrak{g}^L)} H\sharp U(\mathfrak{g}^R), && \Delta_L(f\sharp \hat y) = f_{(1)} \sharp 1_{U(\mathfrak{g}^R)} \otimes_{U(\mathfrak{g}^L)} f_{(2)} \sharp \hat y,
\\
&\epsilon_L \colon H\sharp U(\mathfrak{g}^R) \to  U(\mathfrak{g}^L), && \epsilon_L(f\sharp \hat y) = \theta(\hat y \btl S^{-1}(f)),
\\
& && \epsilon_L(1_H\sharp \hat y_j) =  \hat{x}_j,
\\
& && \epsilon_L(f\sharp \hat y_j) = \epsilon(f)\hat{x}_j - \langle \hat y_j, f \rangle ,%= \epsilon(f)\hat{x}_j - \langle \hat y_j, f \rangle 1_{U(\mathfrak{g}^L)} ,
\end{align*} 
with (iii) antipode 
\begin{align*}
&\mathcal S \colon H\sharp U(\mathfrak{g}^R) \to H\sharp U(\mathfrak{g}^R), && \mathcal S (f\sharp \hat y) = \hat{y}_{[0]} \cdot S^2(\hat{y}_{[-1]})Sf, 
\\
& && \mathcal S (f \sharp \hat{y}_j) = \sum_i \hat{y}_i \cdot {\mathcal{O}}^i_j Sf = \sum_i {\mathcal{O}}^i_j S(f_{(2)})\sharp (\hat y_i \btl S(f_{(1)})) + \sum_{i} C^i_{ij}  Sf\sharp 1_{U(\mathfrak{g}^R)}, 
\\
& &&  \mathcal S (1_H \sharp \hat{y}_j) 
= \sum_i \hat y_i \cdot \OO^i_j = \alpha_L(\hat x_j) = \sum_i {\mathcal{O}}^i_j \sharp \hat y_i + \sum_{i} C^i_{ij} = \alpha_L(\hat z_j) + \sum_{i} C^i_{ij}, 
\\
& && \mathcal S (\alpha_L(\hat x_j)) = \tau(\sum_i \hat y_i \cdot \OO^i_j) =  \sum_{k,i} \bar{\OO}^i_j\OO^k_i \sharp \hat y_k + \sum_{k,i} \bar{\OO}^i_j C^k_{ki} =\hat y_j + \sum_i C^i_{ij}
\\
& && \mathcal S (\alpha_L(\hat{z}_j)) = \tau(\sum_i \OO^i_j \sharp \hat y_i) =  \sum_{i,k} \hat y_k \cdot \OO^k_i \cdot \bar{\OO}^i_j = \hat y_j, 
\end{align*} 
where in formulas $\hat y \in U(\mathfrak{g}^R)$, $\hat x \in U(\mathfrak{g}^L)$, $f\in H$, and $j\in \{1, \ldots,n\}$. 

\begin{remark} The second formula in (52) and the first formula in (53) in \cite{halg} have a mistake in sign: the minus sign should be replaced by a plus sign, and vice versa. The correct formulas are 
	\begin{equation} \mathcal{S}^2 (\hat y_\mu) = \mathcal S(\hat x_\mu ) = \hat y_\mu - C^\lambda_{\mu\lambda}, \quad \mathcal{S}^{-2}(\hat x_\mu) = \mathcal{S}^{-1}(\hat y_\mu) = \hat x_\mu + C^\lambda_{\mu\lambda}, 
	\end{equation}
	\begin{equation} \mathcal{S}^2 (\hat x_\mu) = \hat x_\mu - C^\lambda_{\mu\lambda}, \quad \mathcal{S}^{-2}(\hat y_\mu)  = \hat y_\mu + C^\lambda_{\mu\lambda},
\end{equation}
since the antipode $\mathcal S$ maps 
$\hat z_\mu = \hat x_\mu + C^\lambda_{\mu\lambda} \mapsto \hat y_\mu \mapsto x_\mu.$ Within this remark, summation over repeated indices is understood.
\end{remark}

\end{document}